\documentclass[a4paper,reqno]{amsart}

\usepackage{graphicx}
\usepackage{mathrsfs}
\usepackage{amsfonts}
\usepackage{amssymb}
\usepackage{amsmath}
\usepackage{amsthm}
\usepackage{amscd}
\usepackage[all,2cell,poly]{xy}
\usepackage{color}
\usepackage[pagebackref,colorlinks]{hyperref}
\usepackage{enumerate}
\usepackage{bm}
\usepackage{bbm}
\usepackage[normalem]{ulem}
\usepackage{mathrsfs}

\usepackage{tikz}
\usetikzlibrary{arrows.meta}

\usepackage{comment}
\usepackage{hyperref}
\UseAllTwocells \SilentMatrices 

\numberwithin{equation}{section}

\usepackage{comment}

\theoremstyle{plain}
\newtheorem{thm}{Theorem}[section]
\newtheorem{prop}[thm]{Proposition}
\newtheorem{cor}[thm]{Corollary}
\newtheorem{lemma}[thm]{Lemma}
\newtheorem{conjn}{Conjecture}
\newtheorem*{conj}{Conjecture}

\theoremstyle{definition}
\newtheorem{deff}[thm]{Definition}
\newtheorem{example}[thm]{Example}
\newtheorem{nota}[thm]{Notation}
\newtheorem{rmk}[thm]{\bf Remark}

\newcommand{\vc}{\operatorname{Vec}}
\newcommand{\fvc}{\operatorname{vec}}

\newcommand{\Rep}{\operatorname{Rep}}

\newcommand{\rr}{r_{\operatorname{ck}}}

\newcommand{\fRep}{\operatorname{rep}}
\newcommand{\ffRep}{\operatorname{rep^{\mathbb Z}}}

\def\a{\alpha}

\newcommand{\K}{\mathsf k}

\def \Z{\mathbb Z}
\def\-{\text{-}}

\newcommand{\Hom}{\operatorname{Hom}}

\newcommand{\gr}{\operatorname{gr}}

\newcommand{\M}{\mathbb M}

\newcommand\Gr[1][]{{\operatorname{{Gr}^{#1}-}}}
\newcommand\fGr[1][]{{\operatorname{{gr}^{#1}-}}}

\newcommand\Modd[1][]{{\operatorname{{Mod}^{#1}-}}}

\newcommand{\fModd}{\operatorname{{mod}-}}

\newcommand{\id}{\operatorname{id}}

\begin{document}

\title{Morita theory of finite representations of Leavitt path algebras}

\author{Wolfgang Bock}
\address{(Bock): 
Mathematics Department, Linnaeus Unversity, Universitetsplatsen 1, V\"axj\"o, 352 52, Sweden}
\email{wolfgang.bock@lnu.se}

\author{Roozbeh Hazrat}
\address{(Hazrat): Centre for Research in Mathematics and Data Science,  
Western Sydney University\\
Australia} \email{r.hazrat@westernsydney.edu.au}

\author{Alfilgen  Sebandal}
\address{(Sebandal):
Research Center for Theoretical Physics\\
Central Visayan Institute Foundation\\
Jagna, Bohol, Philippines
} \email{a.sebandal@rctpjagna.com}

\begin{abstract} 
The Graded Classification Conjecture states that for finite directed graphs $E$ and $F$, the associated Leavitt path algebras $L_\K(E)$ and $L_\K(F)$ are graded Morita equivalent, i.e., $\Gr L_\K(E) \approx_{\gr} \Gr L_\K(F)$,   if and only if, their graded Grothendieck groups are isomorphic $K_0^{\gr}(L_\K(E)) \cong K_0^{\gr}(L_\K(F))$ as order-preserving $\mathbb Z[x,x^{-1}]$-modules. Furthermore, if under this isomorphism, the class $[L_\K(E)]$ is sent to $[L_\K(F)]$ then the algebras are graded isomorphic, i.e., $L_\K(E) \cong _{\gr} L_\K(F)$.
 
 In this note we show that, for finite graphs $E$ and $F$ with no sinks and sources,  an order-preserving $\mathbb Z[x,x^{-1}]$-module isomorphism  $K_0^{\gr}(L_\K(E)) \cong K_0^{\gr}(L_\K(F))$ gives that the categories of locally finite dimensional graded modules of $L_\K(E)$ and $L_\K(F)$ are equivalent, i.e.,  $\fGr[\mathbb{Z}] L_\K(E)\approx_{\gr} \fGr[\mathbb{Z}]L_\K(F).$ We further obtain that the category of finite dimensional (graded) modules are equivalent, i.e., $\fModd L_\K(E) \approx \fModd L_\K(F)$ and $\fGr L_\K(E) \approx_{\gr} \fGr L_\K(F)$.

\end{abstract}

\maketitle




\section{Introduction}

Since the introduction of Leavitt path algebras about 20 years ago, the program of the complete classification of these algebras  have remained open and quite elusive. In parallel to the analytic version of the theory, i.e., graph $C^*$-algebras, one hopes that some forms of $K$-groups would furnish complete invariants for these algebras. 

Leavitt path algebras are naturally $\mathbb Z$-graded rings. For the Leavitt path algebra $L_\K(E)$, associated to a finite graph $E$ with coefficient in a field $\K$, set 
\[\mathcal V^{\gr}(L_\K(E)) :=\big \{ [P] \mid P  \text{ is a graded finitely generated projective right }   L_\K(E)\text{-module} \big \},  \]
where $[P]$ is the class of graded  $L_K(E)$-modules graded isomorphic to $P$. The set $\mathcal V^{\gr}(L_\K(E))$ becomes a commutative monoid with the direct sum addition $[P]+[Q]:=[P\oplus Q]$. Furthermore, there is an action of $\mathbb Z$ on this monoid, defined by ${}^n [P]=[P(n)]$, $n\in \mathbb Z$, where $P(n)$ is the graded module obtained by $P$ shifted by $n$. The group completion of this monoid is the graded Grothendieck group $K_0^{\gr}(L_\K(E))$, which with the induced action of $\mathbb Z$ can be considered as a $\mathbb Z[x,x^{-1}]$-module. It was proved that $\mathcal V^{\gr}(L_\K(E))$ is a cancellative monoid~\cite{arali} and thus can be considered as a  submonoid of $K_0^{\gr}(L_\K(F))$. 
For finite graphs $E$ and $F$, a $\mathbb Z[x,x^{-1}]$-module homomorphism  $\phi: K_0^{\gr}(L_\K(E)) \rightarrow K_0^{\gr}(L_\K(F))$, sending $\mathcal V^{\gr}(L_\K(E))$ to $\mathcal V^{\gr}(L_\K(F))$ is called an order-preserving homomorphism. 

The Graded Classification Conjecture~(\cite{mathann, hazdyn},  \cite[\S 7.3.4]{lpabook}) states that for finite directed graphs $E$ and $F$, the associated Leavitt path algebras are graded Morita equivalent, 
 if and only if, their graded Grothendieck groups are isomorphic, $K_0^{\gr}(L_\K(E))\cong K_0^{\gr}(L_\K(E))$,  as order-preserving $\mathbb Z[x,x^{-1}]$-modules.  Furthermore, if under this isomorphism, the class $[L_\K(E)]$ is sent to $[L_\K(F)]$ then the algebras are graded isomorphic, i.e., $L_\K(E) \cong _{\gr} L_\K(F)$.
 We refer the reader to \cite{abconj, arapardo,guidohom,guidowillie,bilich, carlsen,eilers,vas} for works on the graded classification conjecture and \cite{willie} for a comprehensive survey. 
 
 Using the deep work of George Bergman on realisation of conical monoids~\cite{bergman74}, the monoid $\mathcal V^{\gr}(L_\K(E))$ can be described directly from the graph $E$. Namely, there is a $\mathbb Z$-monoid isomorphism 
 $\mathcal V^{\gr}(L_\K(E))\cong T_E$, where 
 \begin{equation*}
T_E= \Big \langle \, v(i), v \in E^0, i \in \mathbb Z  \, \,  \Big \vert \, \, v(i)= \sum_{v\rightarrow u\in E^1} u(i+1) \ \Big \rangle. 
\end{equation*}
Here $E^0$ is the set of vertices,  $E^1$ the set of edges of the graph $E$, and the relations in the monoid are defined for those vertices that emit edges.  The action of $\mathbb Z$ on $T_E$ is defined by ${}^n v(i)=v(i+n)$, where $n\in \mathbb Z$.  Under this isomorphism, the class $[L_\K(E)]$ is sent to $1_E:= \sum_{v\in E^0} v$.

The so-called \emph{talented monoid} $T_E$ of the graph $E$, was introduced in \cite{hazli}. It is believed (and investigated in \cite{hazli}) that there is a beautiful and close relation between the geometry of a graph $E$ and the monoid structure of the talented monoid $T_E$ parallel to the correspondence between the algebraic structure of $L_\K(E)$ and the geometry of the graph $E$~\cite{lpabook}. We refer the reader to \cite{Luiz,lia,alfi, alfi2} for further work on the talented monoids.

The graded classification theorem can then be restated in the language of the monoids.

\begin{conj}[Graded Classification Conjecture]\label{conj1}
Let $E$ and $F$ be finite graphs and $T_E$ and $T_F$ their associated talented monoids, respectively. 

\begin{enumerate}[\upshape(1)]

\item There is a $\mathbb Z$-monoid isomorphism $\phi: T_{E} \rightarrow  T_{F}$, if and only if, $L_\K(E)$ is graded Morita equivalent to $L_\K(F)$, i.e., $\Gr L_\K(E) \approx_{\gr} \Gr L_\K(F)$.

\item There is a $\mathbb Z$-monoid isomorphism $\phi: T_{E} \rightarrow  T_{F}$, with $\phi(1_E)=1_F$, if and only if, $L_\K(E)$ is graded isomorphic to $L_\K(F)$ as $\K$-algebras. 

\end{enumerate}
\end{conj}

In this paper, in order to approach the conjectures, instead of category of graded modules $\Gr L_\K(E)$, we concentrate on the category of locally finite graded modules $\fGr[\mathbb{Z}] L_\K(E)$ and also the category of finite dimensional (graded) modules $\fModd L_\K(E)$ and $\fGr L_\K(E)$, i.e., the $L_\K(E)$-modules which are finite dimensional as $\K$-vector spaces.  The main result of the paper shows that if there is a $\mathbb Z$-monoid isomorphism $T_E\cong T_F$ between the talented monoids of the finite graphs with no sources $E$ and $F$, then there is an equivalence of categories  $\fModd L_\K(E)\approx \fModd L_\K(F)$ as well as the graded equivalence of categories $\fGr L_\K(E)\approx_{\gr} \fGr L_\K(F)$. If in addition the graph has no sinks, then $\fGr[\mathbb{Z}] L_\K(E)\approx_{\gr} \fGr[\mathbb{Z}] L_\K(F)$. These are equivalences that commute with auto-equivalence of shifts. The hope is to be able to lift this graded equivalence to $\Gr L_\K(E)\approx_{\gr} \Gr L_\K(F)$ and thus answer the Conjecture above in positive.

The motivation for our investigation is the papers of Ko\c{c} and  \"{O}zayd\i n \cite{koc1,koc2} who characterised finite dimensional representations of the Leavitt path algebras of row-finite graphs. They showed that these representations arise from the so-called maximal cycles and maximal sinks in a graph (see \S\ref{graphsec}). They identified the category of modules over $L_\K(E)$ with a certain full subcategory of (the classical) representations of the graph $E$. The graph monoid $M_E$ is then used to detect when an algebra has finite representation. 

Our starting point, however, is  Green's realisations of the category of modules over a path algebra $P_\K(E)$ with a certain set of relations $r$, i.e., $\Modd P_\K(E)/\langle r\rangle $,  as the category $\Rep(E,r)$ of representations of the graph $E$ with corresponding relations $r$. Similarly,   we shall be using the identification of the category of graded modules of $P_\K(E)/\langle r\rangle $ with the  category $\Rep(\overline E,\overline r)$ of representations of the \emph{covering} graph $\overline E$ with corresponding relations $\overline r$  (see \cite{green}).

We start by carefully establishing the equivalences  
\begin{equation}\label{gfhdks}
\Modd P_\K(E)/\langle r\rangle \cong \Rep(E,r) \quad \text{ and } \quad 
\Gr P_\K(E)/\langle r\rangle \cong \Rep(\overline E,\overline r).
\end{equation} This allows us to construct, for any $k\in \mathbb Z$, an auto-equivalence functor $\mathcal T_k: \Rep(\overline E,\overline r)\rightarrow \Rep(\overline E,\overline r)$ which commutes with the shift functors on  $\Gr P_\K(E)/\langle r\rangle$. This is done in order to show that the categorical equivalences we obtain for the graded modules are ``graded'' equivalences.  These types of categorical equivalences are important ingredients in the setting of combinatorial algebras which relate these algebras with the dynamical invariants of the graphs and thus symbolic dynamics (see~\cite{willie,hazdyn}).

Next we specialise (\ref{gfhdks}) to the case of Leavitt path algebras by choosing the graph to be $\hat E$, the double graph of $E$, and relations $r$ to be the Cuntz-Krieger relations (Definition~\ref{llkas} (3) and (4)) denoted by $\rr$ and thus obtaining  $\Modd L_\K(E) \cong \Rep(\hat E,\rr)$ and 
$\Gr L_\K(E) \cong \Rep(\overline {\hat E},\overline \rr)$. Translating the finite representations of $L_K(E)$ into the finite representations of the graphs, 
we will then see that the components of the graphs that give finite representations are exactly the maximal cycles and maximal sinks in the graph. We then show that if the talented monoids $T_E$ and $T_F$ are $\mathbb Z$-isomorphic, then there is a one-to-one correspondence between maximal cycles and maximal sinks of the graphs which preserves the length of the cycles. Consequently, we can establish the equivalence of the categories of finite and locally-finite (dimensional) representations of graphs, i.e., $\fRep(\hat E,\rr)\cong \fRep(\hat F,\rr)$ and 
$\fRep(\overline {\hat E}, \overline \rr)\cong \fRep(\overline {\hat F},\overline \rr)$ by a careful analysis. Combining with the equivalences (\ref{gfhdks}) in the case of Leavitt path algebras, we obtain our main results that a $\mathbb Z$-monoid isomorphism $T_E\cong T_F$ gives an equivalence of categories in Theorem~\ref{mainthfer}. 
 

Here again the ``talent'' of the monoid $T_E$ signifies itself. Namely  consider the graphs 
\begin{equation*}
{\def\labelstyle{\displaystyle}
E :   \xymatrix{
 \bullet \ar[r] & \bullet 
}}
\qquad \quad \quad \quad
{\def\labelstyle{\displaystyle}
F :   
\xymatrix{
 \bullet \ar[r] & \bullet \ar@(ur,rd)\\
}} 
\end{equation*}
Then there is an isomorphism of graph monoids $M_E\cong M_F$. However the category of finite dimensional (graded) modules of $L_\K(E)$ and $L_\K(F)$ are not equivalent (see also Example~\ref{whenmy}).  

The results of the paper can be possibly expanded in different directions, such as row-finite graphs, graphs with infinite emitters as well as for weighted Leavitt path algebras. However we decided to restrict ourselves to finite graphs as the proofs are more transparent and the original graded classification conjecture stated and still open for this case. Throughout the paper the set of natural numbers $\mathbb N$ contains zero.

\section{preliminary}

\subsection{Graphs}\label{graphsec}
Suppose that $E$ is a {\it directed graph} (which we shall simply call a \emph{graph}), i.e. a quadruple $E=(E^{0}, E^{1}, r, s)$, where $E^{0}$ and $E^{1}$ are
sets and $r,s$ are maps from $E^1$ to $E^0$. The elements of $E^0$ are called \textit{vertices} and the elements of $E^1$ \textit{edges}. We may think of an edge $e\in E^1$ as an arrow from $s(e)$ to $r(e)$. 
We assume that $E$ is \emph{row-finite}, i.e. for each vertex $v\in E^0$ there are at most finitely many edges in $s^{-1}(v)$. 

Let $v\in E^0$ and let $c$ and $d$ the cardinal of the set $s^{-1}(v)$ of edges it \emph{emits}, and the set $r^{-1}(v)$ of edges it \emph{receives}, respectively. We say that $v$ is a \emph{sink} if $c=0$,  and  regular otherwise. We denote the set of sinks in $E$ by $\textnormal{Sink}(E)$. We say that $v$ is a \emph{source} if $d=0$. A vertex which is both a sink and a source is called an \emph{isolated vertex}.

We use the convention that a (finite) path $p$ in $E$ is
a sequence $p=\a_{1}\a_{2}\cdots \a_{n}$ of edges $\a_{i}$ in $E$ such that
$r(\a_{i})=s(\a_{i+1})$ for $1\leq i\leq n-1$. We define $s(p) = s(\a_{1})$, and $r(p) =
r(\a_{n})$ and denote the length of the path, the number of edges composing $p$, by $|p|$. We denote set of vertices in the path $p$ by $p^0$ and the set of all paths in $E$ by $\textnormal{Path}(E)$. 
We say that a vertex $u$ \emph{connects to a vertex} $v$, and  write $u\ge v$, if there is a path $\alpha$ with $s(\alpha)=u$ and $r(\alpha)=v$. We say that $u$ \emph{connects to a path} $\alpha$ if there exists $v\in\alpha^0$ such that $u$ connects to $v$.
A \emph{closed path based} at a vertex $v\in E^0$ is a path $\alpha$ with $|\alpha|\ge 1$ such that $v=s(\alpha)=r(\a)$. The closed path $\alpha$ based at $v$ is said to be $simple$ if there is only one edge in $\alpha$ whose source is $v$. Denote by  $\textnormal{CSP}(v)$ the set of all such paths. Let $\text{CSP}(E)_{>1}=\{v\in E^0 \mid |\textnormal{CSP}(v)|>1 \}$. A \emph{cycle} is a closed path such that $s(e_i) \neq s(e_j)$ for every $i \neq j$. A \emph{loop} is a cycle of length one. We say a graph is \emph{acyclic} if it has no cycles. 

We define the \emph{tree of a vertex} $v\in E^0$ as 
\[T(v)=\{ w \in E^0 \mid v \geq w \}\] and for a subset $S\subseteq E^0$ as 
\[T(S)=\{ w \in E^0 \mid v \geq w, \text{ for some } v\in S\}.\]

One can define a pre-ordering on the set of sinks and cycles of a graph $E$.  We say that a cycle $C$ \emph{connects} to a sink $z$  if $z\in T(C)$, i.e., if there is a path from $C$ to $z$.  Similarly a cycle $C$ \emph{connects} to a cycle $D$, if $T(C) \cap D^0 \not = \varnothing$, i.e., there is path from $C$ to $D$.  This defines a pre-ordering on the set of cycles and sinks. Furthermore, this pre-ordering is a partial order if and only if the cycles in $E$ are mutually disjoint.  With this pre-ordering, a cycle $C$ is maximal if no other cycle connects to $C$ (in particular, a maximal cycle is disjoint from all other cycles). A sink $z$ is maximal if there is no cycle $C$ which connects to $z$.
We denote by $C(E)$ the set of vertices in $E$ lying in some maximal cycle, that is,   $C(E)=\{v\in E^0\mid v\in C^0, \textnormal{~for~some~maximal~cycle~} C\} $.

\subsection{Monoids}

Recall that a {\it monoid} is a semigroup with an identity element. Throughout this paper monoids are commutative, written additively, with the identity element denoted by $0$.  A monoid homomorphism $\phi:M\rightarrow N$ is a map between monoids $M$ and $N$ which respects the structures and $\phi(0)=0$. Every commutative monoid $M$ is equipped with a natural pre-ordering: $n \leq m$ if $n+p=m$ for some $p\in M$. We write $n < m$ if $m=n+p$ for some $p \not = 0$.

For a graph $E$, the \emph{graph monoid} $M_E$ is defined as a free commutative monoid over the vertices of $E$,  subject to identifying  vertices emitting edges with the sum of vertices they are connected to:
\begin{equation*}
M_E= \Big \langle \, v \in E^0 \, \, \Big \vert \, \,  v= \sum_{v\rightarrow u \in E^0} u \, \Big \rangle. 
\end{equation*}

The graph monoids were introduced by Ara, Moreno and Pardo~(\cite{ara2006}, \cite[\S 6]{lpabook}) in relation with the theory of Leavitt path algebras. It was shown that the group completion of $M_E$ is the Grothendieck group $K_0(L_\K(E))$, where $L_\K(E)$ is the Leavitt path algebra with coefficient in a field $\K$, associated to $E$.

Let $M$ be a commutative monoid with an abelian group $\Gamma$ acting on it via monoid automorphisms, i.e., a \emph{$\Gamma$-monoid}.  For $\alpha \in \Gamma$ and $a\in M$, we denote the action of $\alpha$ on $a$ by ${}^\alpha a$. 
A monoid homomorphism $\phi:M \rightarrow N$ is called $\Gamma$-\emph{monoid homomorphism} if $\phi$ respects the action of $\Gamma$, i.e., $
\phi({}^\alpha a)={}^\alpha \phi(a)$.

A $\Gamma$-\emph{order-ideal} of a $\Gamma$-monoid $M$ is a  subset $I$ of $M$ such that for any $\alpha,\beta \in \Gamma$, ${}^\alpha a+{}^\beta b \in I$ if and only if 
$a,b \in I$. Equivalently, a $\Gamma$-order-ideal is a submonoid $I$ of $M$ which is closed under the action of $\Gamma$ and it  is
\emph{hereditary} in the sense that $a \le b$ and $b \in I$ imply $a \in I$. The set $\mathcal L(M)$ of $\Gamma$-order-ideals of $M$ forms a (complete) lattice. We say $M$ is a \emph{simple} 
$\Gamma$-\emph{monoid} if the only $\Gamma$-order-ideals of $M$ are $0$ and $M$.

In this paper we will be working with the so-called talented monoid $T_E$, associated to a graph $E$,  which can be considered as the ``time evolution model'' of the graph monoid $M_E$.  The notion of talented monoids allows us to write the classification conjectures for Leavitt and $C^*$-algebras in a unified manner (see Conjecture~\ref{conji1}).  The positive cone of the graded Grothendieck group of a Leavitt path algebra $L_\K(E)$ can
be described purely based on the underlying graph, via the talented monoid $T_E$ of $E$~(Section~\ref{lpat} and \cite{hazli}). The benefit of talented monoids is that they give us control over elements of the monoids (such as minimal elements, atoms, etc.)
and consequently on the ``geometry'' of the graphs such as the number of cycles, their lengths, etc. (\cite{hazli,Luiz}).

\begin{deff}\label{talentedmon}
Let $E$ be a row-finite graph. The \emph{talented monoid} of $E$, denoted $T_E$, is the commutative 
monoid generated by $\{v(i) \mid v\in E^0, i\in \mathbb Z\}$, subject to
\[v(i)=\sum_{e\in s^{-1}(v)}r(e)(i+1)\]
for every $i \in \mathbb Z$ and every regular vertex $v\in E^{0}$. The additive group $\mathbb{Z}$ of integers acts on $T_E$ via monoid automorphisms by shifting indices: For each $n,i\in\mathbb{Z}$ and $v\in E^0$, define ${}^n v(i)=v(i+n)$, which extends to an action of $\mathbb{Z}$ on $T_E$. Throughout we will denote elements $v(0)$ in $T_E$ by $v$. When the graph $E$ is finite, we denote $1_E:=\sum_{v\in E^0} v \in T_E$. We say a monoid isomorphism $\phi: T_E \rightarrow T_F$ between the graphs $E$ and $F$ is a $\mathbb Z$-isomorphism, if $\phi$ respects the $\mathbb Z$-action. Furthermore, $\phi$ is called  \emph{pointed} if $\phi(1_E)=1_F$. 
\end{deff}

The crucial ingredient for us is the action of $\mathbb Z$ on the monoid $T_E$. The general idea is that the monoid structure of $T_E$ along with the action of $\mathbb Z$ resemble the graded ring structure of a Leavitt path algebra $L_{\mathsf k}(E)$.
Thus Conjecture~\ref{conji1} mentioned in the introduction roughly state that isomorphism of talented monoids can lift to the isomorphisms of graded and equivariant $K$-theory of respected Leavitt and graph $C^*$-algebras.

\begin{deff}\label{defcuc}
The monoid $T=\bigoplus_{i=1}^k \mathbb N$ with the action of $\mathbb Z$ defined by ${}^1(a_1,\dots,a_{k-1},a_k)=(a_k,a_1\dots, a_{k-1})$ is called the $\mathbb Z$-\emph{cyclic monoid of period} $k$. The monoid $T=\bigoplus_{\infty} \mathbb N$, with the action of $\mathbb Z$
defined by ${}^1(a_i)_{i\in \Z}= (b _i)_{i\in \Z}$, where $b_i=a_{i+1}$, (i.e., shifting the elements to the left) is called  the \emph{infinite $\mathbb Z$-cyclic} monoid. 
\end{deff}

Recall that a graph is called \emph{acyclic} if it has no cycle. A \emph{comet} graph is a graph which consists of only one cycle (and all other vertices connects to this unique cycle)~\cite{lpabook}. 

\begin{example}\label{cometex}
(1) Let $E$ be a finite graph. Then $T_E$ is a $\mathbb{Z}$-cyclic monoid of period $k$, i.e., $T_E \cong \bigoplus_{i=1}^k \mathbb N$ as $\mathbb Z$-monoids 
 if and only if $E$ is a comet graph with the unique cycle of length $k$.

(2) Let $E$ be a finite graph. Then $T_E$ is an infinite $\mathbb{Z}$-cyclic monoid 
 if and only if $E$ is an acyclic graph with a unique sink.  
 
\end{example}
\begin{proof} (1)
     Let $E$ be a finite graph such that $T_E=\bigoplus_{i=1}^k\mathbb{N}$ is a $\mathbb{Z}$-cyclic monoid. First we show $T_E$ is simple $\mathbb{Z}$-monoid. Suppose $0 \not = I \subseteq T_E$ is a $\mathbb{Z}$-order-ideal. Then there is $a=(a_0,\dots,a_{k-1})\in I$, where $a_i \not = 0$, for some $0\leq i \leq k-1$. Since $e_i \leq a$, where $e_i$ is a vector where all entries are zero except $1$ in the $i$-th component and $I$ is a $\mathbb{Z}$-order-ideal, then $e_i\in I$. Since $I$ is closed under the action of $\mathbb Z$, then for any $j\in \mathbb{Z}$,  ${}^je_i=e_{i+j \mod k} \in I.$ This implies $I=T_E$ and $T_E$ is a simple $\mathbb{Z}$-monoid. 
     
     On the other hand, for all $x\in T_E$, we have $^kx=x$. By \cite[Proposition 4.2(i)]{hazli}, $E$ has a cycle of length $k$ with no exit. If $E$ has another cycle, then this cycle has to have no exits (again by \cite[Proposition 4.2(ii)]{hazli}). The hereditary and saturated subsets generated by these two cycles give two distinct $\mathbb{Z}$-order-ideals in $T_E$ which is a contradiction
       with the fact that $T_E$ is simple. 
     Hence, $E$ has one cycle which has no exits. If $E$ has a sink $x$, then $^mx\neq x$ for all $m\in \mathbb Z$, a contradiction. Thus, every vertex in $E$ is connected to a unique cycle which implies that $E$ is a comet graph. 

     Conversely, suppose $E$ is a comet graph with a unique cycle of length $k$. Since source removal would preserve the talented monoid (\cite[Proposition 4.4]{Luiz}), we may assume that $E$ is a cycle. Let $E^0=\{v_0,v_1,\cdots , v_{k-1}\}$ with $r(s^{-1}(v_i))=\{v_{i+1}\}$ for $i<k-1$ and $r(s^{-1}(v_{k-1}))=\{v_{0}\}$. Then we have
\begin{eqnarray*}
    T_E &=& \left \langle v_i(j)\mid 0\leq i\leq k-1, j\in \mathbb{Z}\, \Bigl\vert \, 
    v_i(j)= \begin{cases}
  v_{i+1}(j+1), & 0\leq i< k-1\\
  v_0(j+1), & i=k-1
\end{cases}
    \right \rangle.
\end{eqnarray*}

Now it is easy to see $T_E \cong \bigoplus_{i=1}^k\mathbb{N}$ as $\mathbb Z$-monoids.

\noindent (2) Suppose $T_E= \bigoplus_\infty \mathbb N$, as in Definition~\ref{defcuc}. A similar argument as in part (1) shows that $T_E$ is a simple $\mathbb Z$-monoid. Observe that $^nx$ and $x$ are not comparable for any $n\in \mathbb{Z}$ and $x\in T_E$. Hence, by \cite[Proposition 4.2]{hazli},  $T_E$ has no cycles which also implies that $E$ also has a sink, as $E$ is finite. If $E$ has more than one sinks, then the  hereditary and saturated subsets generated by these two sinks give two distinct $\mathbb{Z}$-order-ideals in $T_E$ which is a contradiction
       with the fact that $T_E$ is simple.
 Thus, $E$ has a unique sink.

Conversely, suppose $E$ is acyclic with a unique sink. Again, we may assume that $E$ is just one isolated vertex $v$ by source removal. Hence, $T_E=\langle v(i) \mid i\in \mathbb Z\rangle$. Define $\varphi : T_E\rightarrow \bigoplus_\infty \mathbb N$ by $v(i)\mapsto e_i$ where $e_i$ has $1$ in the $i$th entry and $0$ elsewhere. Now, $\varphi(^nv(i))=\varphi (v(i+n))=e_{i+n}={}^ne_i={}^n\varphi (v(i))$ and $\varphi$ is clearly a bijective $\mathbb Z$-homomorphism. 
\end{proof}

The following Lemma is needed in the next section to determine the (locally) finite representations of the Leavitt path algebras. 

\begin{thm}\label{talmax}
    Let $E$ be a finite graph and $T_E$ its associated talented monoid.  
    \begin{enumerate}[\upshape(1)]
\item  There is a one-to-one correspondence between maximal cycles of $E$ and $\mathbb Z$-order-ideals $I$ of $T_E$ with $T_E/I$  a $\mathbb Z$-cyclic monoid of finite period. Under this correspondence,  the length of the cycle is equal to the period of the corresponding monoid $T_E/I$.

\item  There is a one-to-one correspondence between maximal sinks of $E$ and $\mathbb Z$-order-ideals $I$ of $T_E$ with $T_E/I$ an infinite $\mathbb Z$-cyclic monoid.

\end{enumerate}

       \end{thm}
    \begin{proof}
        
        (1) Let $C$ be a maximal  cycle in $E$. Let $H=\{ v\in E^0 \mid T(v) \cap C^0=\varnothing \}.$ The set $H$ is a hereditary and saturated subset of $E$ (the empty set if and only if $E$ is a comet, i.e., it consists of one unique cycle--the cycle $C$-- and any vertex connects to this cycle). Consider the $\mathbb Z$-order-ideal $I$ generated by $H$ in $T_E$. We show that $T_E/I$ is a $\mathbb Z$-cyclic monoid.  We have $T_E/I \cong T_{E/H}$ as $\mathbb Z$-monoids. Since the cycle $C$ is maximal, it is easy to see that  $E/H$ is a comet graph. Thus $T_E/I \cong \bigoplus_{i=1}^k \mathbb N$, where $k$ is the length of the cycle $C$ (see Example~\ref{cometex}). Conversely, suppose $I$ is a $\Z$-order-ideal of $T_E$ with $T_E/I$ cyclic of period $k$. Then $I$ determines a hereditary and saturated subset $H$ such that $H$ generates $I$.  Thus $T_E/I \cong T_{E/H}$. By Example~\ref{cometex}, the graph $E/H$ is a comet graph with its cycle of length $k$. By the construction of $E/H$, the cycle in $E/H$ is a maximal cycle of $E$. One can easily check that this correspondence is a one-to-one correspondence which completes (1). 
        
        (2) Let $z$ be a maximal sink in $E$. Let $H=\{ v\in E^0 \mid  z \not \in T(v)  \}.$ The set $H$ is a hereditary and saturated subset of $E$ (the empty set if and only if $E$ is a acyclic with the unique sink $z$). Consider the $\mathbb Z$-order-ideal $I$ generated by $H$ in $T_E$. We show that $T_E/I$ is an infinite cyclic monoid.  We have $T_E/I \cong T_{E/H}$ as $\mathbb Z$-monoids. It is easy to see that  $E/H$ is an acyclic graph with the unique sink $z$. Thus $T_E$ is an infinite cyclic monoid (see Example~\ref{cometex}). Conversely, suppose $I$ is a $\Z$-order-ideal of $T_E$ with $T_E/I$ an infinite cyclic monoid. Then $I$ determines a hereditary and saturated subset $H$, such that $I=\langle H \rangle$.  Thus $T_E/I \cong T_{E/H}$. By Example~\ref{cometex}, the graph $E/H$ is an acyclic graph with a unique sink.  By the construction of  $E/H$, this sink is maximal in the graph $E$. One can easily check that this correspondence is a one-to-one correspondence which completes (2).
  \end{proof}

\subsection{Leavitt path algebras}\label{lpat}
Leavitt path algebras are $\mathbb Z$-graded rings and the $\mathbb Z$-grading plays a cruicial role in this paper. 
We refer the reader to~\cite{hazi} for the basics of graded ring theory. Throughout the paper rings will have identities and modules are unitary in the sense that the identity of the ring acts as the identity operator on the module. 

Let $\Gamma$ be an abelian group.
A ring $R$ is \emph{$\Gamma$-graded} if $R$ decomposes as a direct sum $\bigoplus_{\gamma\in \Gamma} R_\gamma$, where each $R_\gamma$ is an additive subgroup of $R$ and $R_\gamma R_\eta \subseteq R_{\gamma+\eta}$, for all $\gamma,\eta\in \Gamma$. If $R$ is a $\K$-algebra, where $\K$ is a field, then $R$ is called \emph{$\Gamma$-graded algebra} if the image of $\K$ in $R$ lands in $R_0$, where $0$ is the neutral element of $\Gamma$. Therefore, the components $R_\gamma$, $\gamma \in \Gamma$, become $\K$-vector spaces.

Let $R$ be a $\Gamma$-graded ring. A right $R$-module $M$ is \emph{$\Gamma$-graded} if $M$ decomposes as a direct sum $\bigoplus_{\gamma \in \Gamma} M_\gamma$, where each $M_\gamma$ is an additive subgroup of $M$ and $M_\gamma R_\eta \subseteq M_{\gamma+\eta}$, for all $\gamma, \eta\in \Gamma$.
A graded left $R$-module is defined analogously. 
If $M$ is a $\Gamma$-graded $R$-module and $\gamma\in \Gamma$, then the \emph{$\gamma$-shifted} graded right $R$-module is $M(\gamma):= \bigoplus_{\eta\in \Gamma} M(\gamma)_\eta$, where $M(\gamma)_\eta = M_{\gamma+\eta}$, for all $\eta\in \Gamma$.

We write $\Modd R$ for the category of unitary right $R$-modules and module homomorphisms. If $R$ is $\Gamma$-graded, we denote by $\Gr R$ the category of unitary graded right $R$-modules with morphisms preserving the grading. 

For $\Gamma$-graded rings $R$ and $S$, a functor $\mathcal F \colon \Gr R \to \Gr S$ is called \emph{graded} if $\mathcal F(M(\gamma))=\mathcal F(M)(\gamma)$, for all graded $R$-module $M$ and
$\gamma \in \Gamma$.  For $\gamma\in\Gamma$, the \emph{shift functor}
\begin{equation}\label{shiftshift}
\mathcal{T}_{\gamma}\colon \Gr R\longrightarrow \Gr R,\quad M\mapsto M(\gamma)
\end{equation}
is an isomorphism with the property that $\mathcal{T}_{\gamma}\mathcal{T}_{\eta}=\mathcal{T}_{\gamma+\eta}$, for $\gamma,\eta\in\Gamma$.

For a $\Gamma$-graded ring $R$, recall that $K_0^{\gr}(R)$ is the \emph{graded Grothendieck group}, which is the group completion of the  monoid consisting of the graded isomorphism classes of finitely generated $\Gamma$-graded  projective $R$-modules equipped with the direct sum operation~\cite[Section 3.1.2]{hazi}. Explicitly, let $[P]$ denote the class of graded right $R$-modules graded isomorphic to $P$. Then   
the set 
\begin{equation}\label{zhongshan}
\mathcal V^{\gr}(R)=\big \{ \, [P] \mid  P  \text{ is a finitely generated graded projective R-module} \, \big \}
\end{equation}
 with addition $[P]+[Q]=[P\bigoplus Q]$ forms a commutative monoid.

The monoid $\mathcal V^{\gr}(R)$ is a $\Gamma$-monoid and the abelian group $K_0^{\gr}(R)$ is a $\mathbb Z[\Gamma]$-module with a structure induced from the $\Gamma$-monoid structure on the monoid. More specifically, if $\gamma \in \Gamma$, then for $[P]\in \mathcal V^{\gr}(R)$ we have that $\gamma \cdot [P] := [P(\gamma)]$ defines the action of $\Gamma$.
The class of $R$ itself defines an order unit in $K_0^{\gr}(R)$, and if $S$ is also a $\Gamma$-graded ring, we say that a homomorphism $K_0^{\gr}(R) \rightarrow K_0^{\gr}(S)$ is \emph{pointed (or unital)} if it takes $[R]$ to $[S]$.
Of particular interest to us is the case when $\Gamma = \mathbb Z$ in which case $K_0^{\gr}(R)$ is a $\mathbb Z[x,x^{-1}]$-module. Note that any unital graded homomorphism $\theta : R\to S$ of $\mathbb Z$-graded rings $R$ and $S$, induces a pointed order-preserving $\mathbb Z[x,x^{-1}]$-module homomorphism $K_0^{\gr}(R)\to K_0^{\gr}(S)$.

\begin{deff}\label{defgrrep}
Let $R$ be a $\K$-algebra, where $\K$ is a field. A right $R$-module $M$ is called a \emph{finite representation} of $R$ (or \emph{a finite dimensional $R$-module}) if $M$ has a finite dimension as a $\K$-vector space. If $R$ is a $\Gamma$-graded $\K$-algebra, then a graded right $R$-module $M=\bigoplus_{\gamma\in \Gamma}M_\gamma$ is called a $\Gamma$-\emph{locally finite representation} of $R$ (or \emph{a $\Gamma$-locally finite dimensional graded $R$-module})  if $M_\gamma$ has a finite dimension as a $\K$-vector space, for any $\gamma \in \Gamma$. If $M$  has finite dimension, then $M$ is called a \emph{finite graded representation}. The graded $\K$-algebra  $R$ is called \emph{locally finite}, if it is locally finite as a module over itself (see \cite[Definition~4.2.15]{lpabook}. 
\end{deff}

The full subcategory of $\Modd R$ consisting of finite-dimensional modules is denoted by  $\fModd R$.
Similarly, the full subcategory of $\Gr R$ consisting of locally finite dimensional graded modules is denoted by $\fGr[\Gamma]R$. Note that, although $\Gr R$ is a subcategory of $\Modd R$,  the category $\fGr[\Gamma]R$ is not a subcategory of $\fModd R$. 
As an example, although the module $\K[x,x^{-1}]$ is not a finite representation for the Laurent polynomial ring $\K[x,x^{-1}]$, however considering  $\K[x,x^{-1}]$ as a $\mathbb Z$-graded ring, then $\K[x,x^{-1}]=\bigoplus_{i\in \mathbb Z} \K x^i$ is a locally finite graded representation, as $\K[x,x^{-1}]_i=\K x^i$, $i\in \mathbb Z$, is a $1$-dimensional $\K$-vector space. Lastly, we denote by $\fGr R$ the category of finite graded representations. Clearly $\fGr R$ is a full subcategory of $\fModd R$.

The theory of Leavitt path algebras provides us with a rich source of examples of (graded) rings.  Leavitt path algebras are graded von Neumann regular rings, hence, they have very pleasant graded properties~\cite[\S 1.1.9]{hazi}. Among the class of graded rings one can construct out of these algebras are graded matrix rings, strongly graded rings and crossed products (see~\cite[\S 1.6.4]{hazi}). 

We start with the definition of path algebras. Let $E$ be a graph and $\K$ a field. The \emph{path algebra} $P_\K(E)$ is the free algebra generated by the sets of vertices $E^0$ and edges $E^1$ of $E$ with coefficients in $\K$, subject to the relations $vw=\delta_{vw}v$, for $v,w \in E^0$ and  $s(e)e=e r(e)=e$ for all  $e \in E^1$.

We are in a position to recall the definition of Leavitt path algebras~\cite{lpabook}.

\begin{deff}\label{llkas}
For a row-finite graph $E$ and a  field $\K$, we define the \emph{Leavitt path algebra of $E$},  denoted by $L_\K(E)$, to be the algebra generated by the sets $\{v \mid v \in E^0\}$, $\{ e \mid e \in E^1 \}$ and $\{ e^* \mid e \in E^1 \}$ with the coefficients in $\K$, subject to the relations 

\begin{enumerate}
\item $vw=\delta_{vw}v \textrm{ for every } v,w \in E^0$;

\item $s(e)e=e r(e)=e \textrm{ and }
r(e)e^*=e^*s(e)=e^*  \textrm{ for all } e \in E^1$;

\item $e^* e'=\delta_{e e'}r(e)$ for all $e, e' \in E^1$;

\item $\sum_{\{e \in E^1, s( e)=v\}} e e^*=v$ for every regular vertex $v\in E^0$.

\end{enumerate}
\end{deff}
Relations (3) and (4) are the \emph{Cuntz-Krieger relations}. The elements $e^*$ for $e \in E^1$ are called \emph{ghost edges}. \index{ghost edge} One can show that $L_\K(E)$ is a ring with identity if and only if the graph $E$ is finite (otherwise, $L_\K(E)$ is a ring with local identities, see~\cite{lpabook}).

Setting $\deg(v)=0$ for $v\in E^0$, $\deg(e)=1$ and $\deg(e^*)=-1$ for $e \in E^1$,  we obtain a natural $\mathbb Z$-grading on the free $\K$-ring generated by  $\big \{v,e, e^* \mid v \in E^0,e \in E^1\big \}$. Since the relations in Definition~\ref{llkas} are all homogeneous, the ideal generated by these relations is homogeneous and thus we have a natural $\mathbb Z$-grading on $L_\K(E)$.

The relations between the talented monoids and the Leavitt path algebras come from the fact that $K^{\gr}_0(L_\K(E))$ is the group completion of $\mathcal V^{\gr}(L_\K(E))$ and that there is a $\mathbb Z$-monoid isomorphism 
\[
T_E\cong \mathcal V^{\gr}(L_\K(E)).
\]

The Graded Classification Conjectures for Leavitt path algebras and graph $C^*$-algebras can then be stated as follows~\cite{willie, mathann, hazmark}: 

\begin{conjn}\label{conji1}
Let $E$ and $F$ be finite graphs and $\K$ a field. Then the following are equivalent:

\begin{enumerate}[\upshape(1)]
\item  the talented monoids \(T_E\) and \(T_F\) are \(\mathbb{Z}\)-isomorphic;

\item  the Leavitt path algebras \(L_\K(E)\) and \(L_\K(F)\) are graded Morita equivalent;

\item the graph $C^*$-algebras $C^*(E)$ and $C^*(F)$ are equivariant Morita equivalent. 
\end{enumerate}

Furthermore,  the following are equivalent:

\begin{enumerate}[\upshape(1)]
\item  the talented monoids \(T_E\) and \(T_F\) are pointed \(\mathbb{Z}\)-isomorphic;

\item  the Leavitt path algebras \(L_\K(E)\) and \(L_\K(F)\) are graded isomorphic;

\item the graph $C^*$-algebras $C^*(E)$ and $C^*(F)$ are equivariant isomorphic. 
\end{enumerate}
\end{conjn}

Theorem~\ref{talmax} shows that the talented monoid $T_E$ can completely capture maximal sinks and cycles.  In contrast,  the graph monoid $M_E$ can't be used to describe the (graded) locally finite or irreducible representations of a Leavitt path algebra which are related to maximal cycles and sinks. 
As an example, consider the following graphs:  One can easily check that  $M_E\cong M_F \cong M_G$, but $E$ has two maximal cycles, $F$ has one maximal cycle and one maximal sink, whereas $G$ has two maximal sinks.

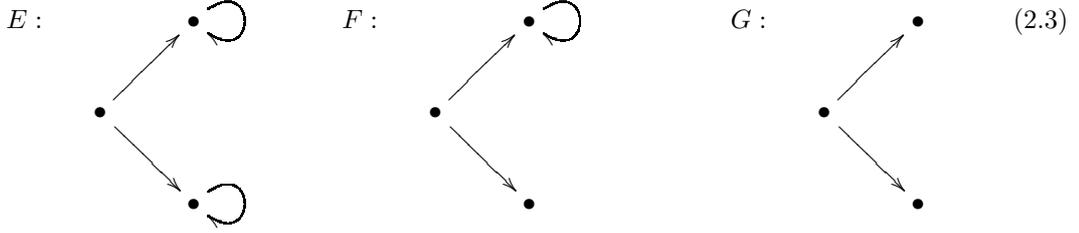
\begin{equation}\label{ggffdd1}
{\def\labelstyle{\displaystyle}
E : \quad \,\,  \xymatrix{
& \bullet \ar@(ur,rd)\\
\bullet  \ar[ru]\ar[dr] & \\
& \bullet \ar@(ur,rd)
}}
\qquad \quad \quad \quad
{\def\labelstyle{\displaystyle}
F : \quad \,\,  
\xymatrix{
& \bullet \ar@(ur,rd)\\
\bullet   \ar[ru]\ar[dr]  &  \\
& \bullet
}} 
\qquad \quad \quad \quad \quad \quad
{\def\labelstyle{\displaystyle}
G : \quad \,\,  \xymatrix{
& \bullet\\
 \bullet    \ar[ru]\ar[dr]  &   \\
 & \bullet
}} 
\end{equation}

\section{Representation of graphs and graph algebras} 
We recall the notion of the representation of a graph with relations according to Green~\cite{green} and then specialise it to the case of Leavitt path algebras. Consider a graph $E$ as a category whose objects are the vertices and for two vertices $v,w$, the morphism set $\Hom_E(v,w)$ are the paths from $v$ to $w$. A covariant functor from the category $E$ to $\vc \K$, the category of vector spaces over a field $\K$, is called a \emph{representation} of $E$. By $\Rep(E)$ we denote the functor category whose objects are the representations of $E$ and morphisms are the natural transformations. Given that we write the paths in a graph from left to right (\S\ref{graphsec}), if $\rho\in \Rep(E)$, and $p=e_1e_2\cdots e_n$ is a path in $E$, then the linear transformation in $\vc \K$ corresponding to path $p$ takes the form $\rho(e_1e_2\cdots e_n)=\rho(e_n)\rho(e_{n-1})\cdots \rho(e_1)$.
The full subcategory of $\Rep(E)$, consisting of all \emph{finite} representations $\rho\in \Rep(E)$, that is, those that  $\bigoplus_{v\in E^0} \rho(v)$ is a finite dimensional $\K$-vector space, is denoted by $\fRep(E)$. Note that if the graph $E$ is finite, then $\fRep(E)$ consists of all functors $\rho:E\rightarrow \fvc \K$, where $\fvc \K$ is the category of finite dimensional vector spaces of $\K$.

\begin{deff}\label{defrel1} Let $E$ be a finite graph. 
A \emph{relation}  in the graph $E$ is a finite formal sum, $\sum_i k_i p_i$, where $k_i\in \K\backslash \{0\}$ and $p_i \in \Hom_E(v,w)$ for fixed vertices $v,w$. Let $r$ be a set of relations in $E$. Denote by $\Rep(E,r)$ the full subcategory of $\Rep(E)$ which satisfies the relations in $r$, i.e., representations $\rho: E \rightarrow \vc \K$ such that 
$\sum_i k_i \rho(p_i)=0$, where $\sum_i k_i p_i$ is a relation in $r$. The subcategory of finite representations with relations are denoted by $\fRep(E,r)$. 
\end{deff}

Let $\langle r \rangle$ be the two-sided ideal of the path algebra $P_\K(E)$ generated by the set of relations $r$ as in Definition~\ref{defrel1}. Define the $\K$-algebra
\begin{equation}\label{brejki}
A_\K(E,r):= P_\K(E)/\langle r \rangle. 
\end{equation}
As observed in~\cite{green}, there exists an exact equivalent covariant functor  
\begin{equation}\label{hng32}
\Phi: \Modd A_\K(E,r) \longrightarrow \Rep(E,r),
\end{equation}
as follows.  Denote the images of $u \in E^0$ and $e\in E^1$ in $A_\K(E,r)$ by the same letters again. 
 For a right $A_\K(E,r)$-module $X$, define the functor $\rho_X: E \rightarrow \vc \K$ by
\begin{align*}
 \rho_X(u)=Xu,&\\
 \rho_X(e:u\longrightarrow v): X u & \longrightarrow X v\\
 xu &\longmapsto xe,
 \end{align*}
 and extend $\rho$ naturally to all paths.   Writing $e=uev$ in the path algebra, shows that the above map is well-defined. One can check that $\rho_X$ satisfies the relations $r$. 
 
 The inverse exact functor 
\begin{equation}\label{hng3277}
\Psi: \Rep(E,r) \longrightarrow \Modd A_\K(E,r),
\end{equation}
is defined as follows. Let $\rho: E \rightarrow \vc \K$. Set 
\[ X=\bigoplus_{u\in E^0} \rho(u).\] Let $i_u$ and $\pi_u$ be the canonical injective and projective homomorphisms  
\[ \rho(u) \stackrel{i_u}{\longrightarrow} \bigoplus_{u\in E^0} \rho(u) \stackrel{\pi_u}{\longrightarrow} \rho(u).\]
If $e:u\rightarrow v$ is an edge in $E$, and $x\in X$, then define 
\[x.e=i_{v}(\rho(e)(\pi_{u}(x))),\] and extend naturally to all paths in $E$.  
Again here $e$ is denoted the image of the path $e$ in the ring $A_\K(E,r)$. 

Recall that $\fRep(E,r)$ is the full subcategory  of all representations $\rho \in \Rep(E,r)$ such that  $\bigoplus_{v\in E^0} \rho(v)$ is a finite dimensional $\K$-vector space.  Denote the full subcategory of the finite representations of the algebra $A_\K(E,r)$ by $\fModd \!A_\K(E,r)$, namely, the collection of all right $A_\K(E,r)$-modules $X$, such that $X$ as a $\K$-vector space is finite dimensional (Definition~\ref{defgrrep}).

Since $E$ is a finite graph, $A_\K(E,r)$ is indeed a unital $\K$-algebra with $\K\rightarrow A_\K(E,r); k\mapsto \sum_{u\in E^0}ku$. Thus if $X$ is an $A_\K(E,r)$-module, then $X = \bigoplus_{u\in E^0} Xu$ as a $\K$-vector space. Thus \[\dim_\K X=\sum_{u\in E^0}\dim_\K Xu. \]

We then have the following proposition (\cite[Theorem 1.1]{green}).

\begin{prop}\label{mainrain}
Let $E$ be a finite graph, $r$ a set of relations in $E$ and let $\Psi: \Rep(E,r) \rightarrow \Modd A_\K(E,r)$ be the equivalence of the categories of~{\upshape (\ref{hng3277})}. Then $\Psi$ restricts to the categories of finite representations $\Psi: \fRep(E,r) \rightarrow \fModd \!A_\K(E,r)$.
\end{prop}

Next we establish an equivalence between the category of graded modules of $A_\K(E,r)$ and the representation of the covering graph of $E$. Although this can be done for an arbitrary $\Gamma$-graded ring, we restrict ourselves to $\mathbb Z$-grading, as in the setting of Leavitt path algebras, this is the canonical grading. Thus throughout this section, our grade group is the integers $\Z$ and the grading on the algebras arise from the weight function $w:E^1 \rightarrow \Z$, assigning $1$ or $-1$ to each edge. We assume that the relations $r$ in the Definition~\ref{defrel1} are homogeneous, i.e., all $p_i$'s in $\sum_i k_i p_i\in r$ have the same degree. Thus $A_\K(E,r)=P_\K (E)/\langle r\rangle$ is a $\mathbb Z$-graded ring.

With integers $\mathbb Z$ as the grade group,  the \emph{covering graph} $\overline E$ (denoted in some literature by $E\times_1 \mathbb Z$) is defined as the graph 
\begin{align}\label{nmehr65}
\overline E^0 = &\big \{v_i \mid v \in E^0, i \in \Z \big \},\\
\overline E^1 = &\big \{ e_{i}:u_{i-1}\longrightarrow v_i \mid e:u \longrightarrow v,  e\in E^1, \deg(e)=1, \text{ and }  i \in \Z \big \} \bigcup \notag\\
& \qquad \qquad \qquad  \big \{ f_{i}:u_{i}\longrightarrow v_{i-1} \mid f:u \longrightarrow v, f\in E^1, \deg(f)=-1, \text{ and } i\in \Z \big \}
. \notag
\end{align}


Let $r$ be a set of relations in a graph $E$.  For a finite path $p=e^1e^2\cdots e^n$ in $E$ and $j\in \mathbb Z$, there is a path $p_j$ in $\overline{E}$ given by
$$p_j = e^1_{j} e^2_{j+1} \cdots e^n_{j+n}.$$
The \emph{lifting} $\overline{r}$ of the relations $r$ in $\overline{E}$ is the set of all relations $\sum_ik_i(q_i)_j$, for every relation $\sum_ik_iq_i\in r$ and every  $j\in \mathbb Z$, where $k_i\in \K$ and $q_i$ are paths in $E$ as described in Definition~\ref{defrel1}.

Our first task is to establish, for any $k\in \mathbb Z$, an auto-equivalence (\emph{a shift functor}) 
\begin{align}\label{autoeqi}
\mathcal T_k:\Rep(\overline E, \overline r) &\longrightarrow \Rep(\overline E,\overline r),\\
\rho &\longmapsto \rho(k)\notag
\end{align} 
which corresponds with the natural shift functor on the category of graded modules of the algebra 
\begin{align}\label{autoeqi1}
\mathcal T_k:\Gr A_\K( E,  r)  &\longrightarrow \Gr A_\K( E,  r),\\
M &\longmapsto M(k).\notag
\end{align}
Let $\rho:\overline E \rightarrow \vc \K$ be a graph representation in $\Rep(\overline E, \overline r)$.  For $k\in \mathbb Z$, define $$\rho(k):=\mathcal T_k(\rho) :\overline E \rightarrow \vc \K,$$ as follows:  
\begin{align}\label{sydairpoit5}
&\rho(k)(u_i)=\rho(u_{i+k}),\\
&\rho(k)(e_i:u_{i-1}\longrightarrow v_i)=\rho(e_{i+k}):\rho(u_{i+k-1}) \longrightarrow \rho(v_{i+k}),\notag\\
&\rho(k)(f_i:u_{i}\longrightarrow v_{i-1})=\rho(f_{i+k}):\rho(u_{i+k}) \longrightarrow \rho(v_{i+k-1}).\notag
\end{align}
Notice that $\rho(k)$ also satisfies the relations $\overline r$ and $\mathcal T_k \mathcal T_{-k} = \id$. 

By Lemma~\ref{mainrain}, there is an equivalence  
\begin{equation}\label{hnggr3277}
\Psi: \Rep(\overline E,\overline r) \longrightarrow \Modd (A_\K(\overline E, \overline r)).
\end{equation}
Thus the shift functor $\mathcal T_k$ on $\Rep(\overline E,\overline r)$ induces a shift functor $\mathcal T_k$ on $\Modd A_\K(\overline E, \overline r)$.


Next we establish an equivalence between the category of graded modules $\Gr A_\K( E,  r)$ with the category of representations of covering graphs $\Rep(\overline E,\overline r)$. An important observation here is that this equivalence is ``graded'', i.e., it preserves the shifts in the respected categories (\ref{autoeqi}, \ref{autoeqi1}). 

Define the  covariant functor 
\begin{equation}\label{hng3211}
\Phi: \Gr A_\K( E,  r) \longrightarrow \Rep(\overline E,\overline r)
\end{equation}
as follows. For a graded right $A_\K( E,  r)$-module $X=\bigoplus_{i\in \mathbb Z} X_i$, define the functor $\rho_X: \overline E \rightarrow \vc \K$ by
\begin{align*}
 \rho_X(u_i)=X_i u,&\\
 \rho_X(e_i:u_{i-1} \longrightarrow v_i):  X_{i-1} u & \longrightarrow X_i v\\
 xu &\longmapsto xe,\\
 \rho_X(f_i:u_{i} \longrightarrow v_{i-1}): X_{i}u &\longrightarrow X_{i-1}v\\
 xu &\longmapsto xf,
 \end{align*}
 and extend $\rho_X$ naturally to all paths. Note that since $\deg(e)=1$ and $\deg(f)=-1$, $xe$ and $xf$ land on the correct component of $X$. 

The inverse exact functor  \begin{equation}\label{hng3266}
\Psi: \Rep(\overline E,\overline r) \longrightarrow \Gr A_\K( E,  r),
\end{equation}
is defined as follows. Let $\rho: \overline E \rightarrow \vc \K$. Set 
\[ X=\bigoplus_{i\in \mathbb Z} X_i,\] where
\begin{equation}\label{sydair3}
X_i  = \bigoplus_{u\in E^0} \rho(u_i).
\end{equation}

For a given $i\in \mathbb Z$, let $i_{u_i}$ and $\pi_{u_i}$ be the canonical injective and projective homomorphisms  
\[ \rho(u_i) \stackrel{i_{u_i}}{\longrightarrow} \bigoplus_{u\in E^0} \rho(u_i) \stackrel{\pi_{u_i}}{\longrightarrow} \rho(u_i).\]
If $e:u\rightarrow v$ is an edge in $E$ with $w(e)=1$, and $x\in X_i$, then define 
\[x.e=i_{v_{i+1}}(\rho(e_{i+1})(\pi_{u_i}(x))).\] Similarly, for $f:u\rightarrow v$ with $w(f)=-1$, and $x\in X_i$, define
\[x.f=i_{v_{i-1}}(\rho(f_{i})(\pi_{u_i}(x))).\]
Again here $e$ and $f$ are denoted the image of the path $e$ and $f$ in the ring $A_\K(E,  r)$. 

Next we will relate locally finite graded  representations of a $\K$-algebra (Definition~\ref{defgrrep}) associated to a graph $E$  to the category of $\Rep(\overline E,\overline r).$

We say a functor $\rho \in \Rep(\overline E,\overline r)$ is a \emph{$\mathbb{Z}$-locally finite representation} of $E$ if for any given $i\in \mathbb Z$, $\rho(u_i)$, $u\in E^0$, is zero for all except a finite number of vertices whose dimension as a $\K$-vector space is finite, i.e., for any given $i \in \mathbb Z$, 
\begin{equation}\label{jfhyt5}
 \dim_\K \big (\bigoplus_{u\in E^0}\rho(u_i)\big) < \infty.
 \end{equation}
 Denote the full subcategory of $\mathbb Z$-\emph{locally finite representations}  of $E$ by $\ffRep(\overline E,\overline r)$ and the locally finite graded representation of the algebra $A_\K(E,r)$ by 
$\fGr[\mathbb{Z}] A_\K(E,r)$. Note that $\fRep(\overline E,\overline r)$ is a full subcategory of $\ffRep(\overline E,\overline r)$ consisting of those functors $\rho$ such that $\dim_\K(\bigoplus_{u\in E^0, i\in \mathbb Z}  \rho(u_i)) < \infty$.

\begin{prop}\label{mainlemma1}
Let $E$ be a graph, $r$ a set of homogeneous relations in $E$ and let $\mu: \Rep(\overline E,\overline r) \rightarrow \Gr A_\K(E, r)$ be the equivalence of the categories of~{\upshape (\ref{hng3266})}. Then $\Psi$ commutes with the shift functors $\mathcal T_k$, $k\in \Z$:
\begin{equation}\label{sfox2}
\xymatrix{
  \Rep(\overline E,\overline r)  \ar[r]^{\Psi} \ar[d]_{\mathcal T_k} & \Gr A_\K( E,  r) \ar[d]^{\mathcal T_k}\\
 \Rep(\overline E,\overline r)  \ar[r]^{\Psi}  & \Gr A_\K( E,  r). 
} 
\end{equation}
Furthermore, $\Psi$ restricts to the categories of locally finite representations 
$\psi: \ffRep(\overline E,\overline r) \rightarrow \fGr[\mathbb{Z}] \!A_\K(E, r)$.  
\end{prop}
\begin{proof}
Let $\rho \in  \Rep(\overline E,\overline r)$. We show that $\mathcal T_k \Psi (\rho)= \Psi\mathcal T_k (\rho)$. On one hand, $\mathcal T_k \Psi (\rho) = \mathcal T_k (X)=X(k),$ where by (\ref{sydair3}), for any $i\in \mathbb Z$, $X_i  = \bigoplus_{u\in E^0} \rho(u_i)$, and
$$X(k)_i=X_{k+i}= \bigoplus_{u\in E^0} \rho(u_{i+k}).$$
On the other hand, 
 $\Psi\mathcal T_k (\rho)=\Psi(\rho(k))=Y,$ where by (\ref{sydair3}) and (\ref{sydairpoit5}), for any $i\in \mathbb Z$,
 $$Y_i= \bigoplus_{u\in E^0} \rho(k)(u_{i})=\bigoplus_{u\in E^0} \rho(u_{i+k}).$$
 This shows that the Diagram~\ref{sfox2} is commutative on objects. Similarly, one can observe that the diagram is commutative on morphisms. The fact that $\Psi$ restricts to $\mathbb{Z}$-locally finite representations follow immediately from the definition~(\ref{jfhyt5}).  
\end{proof}

\subsection{Specialising to the case of Leavitt path algebras}\label{lbvyfgjdh}

Let $E$ be a row-finite graph. Consider $\hat E$ to be a graph consisting of vertices and all the real and ghost edges of $E$ (the \emph{double graph} of $E$~\cite[pp.6]{lpabook}).  Since for any $e, e^*\in \hat E^1$, we have $ee^*\in \Hom(s(e),s(e))$ and $e^*e \in \Hom(r(e),r(e))$, the Cuntz-Krieger relations (3) and (4) in Definition~\ref{llkas} can be considered as relations in the graph $\hat E$ as defined in Definition~\ref{defrel1}. We denote the set of relations by $\rr$. 

Since, by construction, the Leavitt path algebra $L_\K(E)$  is the quotient of the path algebra $P_\K(\hat E)$ subject to the above Cuntz-Krieger relations,  we have $A_\K(\hat E, \rr)=L_\K(E)$ and thus by Propositions~\ref{mainrain} and \ref{mainlemma1} we have equivalences of categories 
\begin{equation}\label{xxxyyy}
\begin{split}
\Modd L_\K(E) &\longrightarrow \Rep(\hat E,  \rr) \\
\fModd L_\K(E) &\longrightarrow \fRep(\hat E,  \rr) 
\end{split}
\quad
\begin{split}
 \qquad \Gr L_\K(E) &\longrightarrow \Rep(\overline{\hat E}, \overline{\rr}),\\
 \qquad \fGr[\mathbb{Z}] L_\K(E) &\longrightarrow \fRep^\mathbb{Z}(\overline{\hat E}, \overline{\rr}).\notag
\end{split}
\end{equation}

Throughout this section the set of relations $r$ in the graph $\hat E$ are those of  (3) and (4) in Definition~\ref{llkas}. Furthermore, we continue calling edges coming from $E$ real edges and their doubles (the $*$-edges) ghost edges in $\hat E$. 

Consider  $\rho:\hat E \rightarrow \vc \K$  in  $\Rep(\hat E,  \rr)$. 
Let $v\in E^0$ be a regular vertex with $s^{-1}(v)=\{e_1,\dots,e_k\}$, the set of all real edges emitting from $v$. Consider the maps 
\begin{equation*}
\begin{split} 
f:\rho(v) &\longrightarrow \bigoplus_{1\leq i \leq k}  \rho(r(e_i))\\
x &\longmapsto (\rho(e_1)(x),\dots,\rho(e_k)(x)),
\end{split}
\quad \quad
\begin{split} 
g: \bigoplus_{1\leq i \leq k} \rho(r(e_i)) &\longrightarrow  \rho(v)\\
(x_1,\dots,x_k)  &\longmapsto \rho(e^*_1)(x_1)+\cdots + \rho(e^*_k)(x_k).
\end{split}
\end{equation*}
We check that $gf=1_{\rho(v)}$:
\begin{multline*}
gf(x)=g(\rho(e_1)(x),\dots,\rho(e_k)(x))=\rho(e^*_1)(\rho(e_1)(x))+\dots+\rho(e^*_k)(\rho(e_k)(x))\\
=\rho(e_1e_1^*+\dots+e_ke_k^*)(x)=\rho(v)(x)=1_{\rho(v)}(x)=x.
\end{multline*}
Similarly, one checks that $fg=1_{\bigoplus_{1\leq i \leq k} \rho(r(e_i))}$. Thus for any regular vertex $v\in \hat E^0$ we have 
\begin{equation}\label{mainprop}
\rho(v)=\bigoplus_{1\leq i \leq k}  \rho(r(e_i)).
\end{equation}
Let $d(v):=\dim_\K \rho(v)$, so that $\rho(v)\cong \K^{d(v)}$, as a $\K$-vector space.  Then the equation~\ref{mainprop} gives 
\begin{equation}\label{mainprop2}
d(v)=\sum_{1\leq i \leq k}  d(r(e_i)). \end{equation}

Clearly if $\rho \in \fRep(\hat E,  \rr)$, then $d(v)$ is finite, for all $v\in E^0$. We will use (\ref{mainprop2}) throughout this section.

\begin{nota}\label{notenan}
 Let $E$ be a graph and $\rho$ a representation in $\fRep(E,r)$. For each $v\in E^0$, denote the dimension $\dim_\K(\rho(v))$ of $\rho(v)$ as a $\K$-vector space by $d^\rho(v)$. That is, 
 \begin{align*}
     d^\rho:E^0 &\longrightarrow \mathbb{N}\\
     v &\longmapsto \dim_\K(\rho(v)).
 \end{align*}
For $H\subseteq E^0$ and $\rho:\overline{E}\rightarrow \fvc \K$ a representation in $\fRep(\overline{E}, \overline{r})$, denote $\rho_H:=\rho|_{ \overline{H }}$, where $\overline{H}:=\{w_i \mid w\in H, i\in \mathbb{Z}\}$.  If $H=\{w\}$, we write $\rho_{\{w\}}:=\rho_w$. For $v\in E^0$, we write 
$$d^\rho_v:= d^{\rho_{v}} :{\{v_i \}_{i\in \mathbb{Z}}}\rightarrow \mathbb{N}.$$ 
Without ambiguity, we shall omit the indication of the representation $\rho$ and we write $d$ and $d_v$ for $d^\rho$ and $d^\rho_v$, respectively.
\end{nota}

We can summarise the above argument as the following corollary. 
\begin{cor}\label{corollarydimension}
    Let $E$ be a finite graph. Then for any representations $\varphi \in rep({\hat{E}}, {r_{ck}}) $ and $\rho \in  rep(\overline{\hat{E}}, \overline{r_{ck}})$, the mappings $d^{\varphi}:E^0\rightarrow \mathbb{N}$ and  $d^{\rho}:\overline{\hat{E}}^0\rightarrow \mathbb{N}$ satisfy \begin{center}$d^\varphi (v) =\sum\limits_{e\in s^{-1}(v)} d^\varphi(r(e))$ ~and~ $d^\rho (v_i) =\sum\limits_{e\in s^{-1}(v)} d^\rho(r(e)_{i+1}),$
\end{center}
for every regular vertex $v\in E^0$ and $i\in \mathbb{Z}$.

\end{cor}

\begin{prop}\label{firstmanin}
Let $E$ be a finite graph such that the sources are isolated vertices.  Then $\rho:\hat E \rightarrow \fvc \K$  in  $\fRep(\hat E,  \rr)$ has the form 
\begin{equation}\label{finrep11}
\rho(v)=
 \left\{ \begin{array}{ll}
 \K^{n_C} & \textrm{if $v\in C$ a maximal cycle},\\
  \K^{n_v} & \textrm{if } v  \textrm{ is an isolated vertex},\\
 0 & \textrm{otherwise,}
  \end{array} \right.
\end{equation}
for fixed $n_C\in \mathbb N$ for each maximal cycle $C$  and $n_v\in \mathbb N$. 
\end{prop}
\begin{proof}

 First note that if $e\in E^1$, then Corollary \ref{corollarydimension} implies that $d(s(e))\geq d(r(e))$. Let $v$ be a vertex on a cycle $C$ and write $C=vc_1c_2\dots c_nv$, where $c_i\in E^1$. Then $d(v) \geq d(r(c_1)) \geq d(r(c_2)) \geq \cdots \geq d(v)$. Therefore all the vertices on the cycle $C$ have the same dimension (either zero, infinite or a non-zero natural number). Since we are dealing with finite representations $\rho$, the dimensions of vertices are all finite. Suppose the cycle $C$ has an exit, i.e., an edge $f$ such that $s(f)\in C$ and $f\not = c_i$, $1 \leq i \leq n$.  Suppose $s(f)=s(c_l)$. Then by Corollary \ref{corollarydimension}, $d(s(c_l)) \geq d(r(c_l))+ d(r(f))$. Since the dimensions on the cycles are finite and equal, $d(r(f))=0$. Next suppose the vertex $v$ is not on a maximal cycle. Thus there is a cycle $B$ with a path from $B$ to $C$.  Since the dimension of exit edges from $B$ has to have dimension zero, an easy induction implies that $d(v)=0$. On the other hand, since there is no relations on the isolated vertices, they can be assigned any finite dimensional $\K$-vector space. Putting these together we conclude that if $\rho \in \fRep(\hat E,\rr)$, then  it has to be of the form of (\ref{finrep11}). 
\end{proof}

\begin{example}\label{whenmy}
Note that two non-isomorphic elements $\rho, \sigma \in \fRep(\hat E,\rr)$ can have the same form described in (\ref{finrep11}). As an example, let $E$ be a graph with one vertex $v$ and one loop $e$ (thus $\hat E$ consists of $v$ and two loops $e,e^*$). Let $\rho(v)=\K^2$ with the linear transformation 
$\rho(e)=\begin{pmatrix}
1 & 1\\
0 & 1\\
\end{pmatrix}:\K^2\rightarrow \K^2$ and $\rho(e^*)=\rho(e)^{-1}.$ Furthermore, 
let $\sigma(v)=\K^2$ with the linear transformation $\sigma(e)=\begin{pmatrix}
1 & 0\\
0 & 1\\
\end{pmatrix}$
and $\sigma(e^*)= \sigma(e)^{-1}$. We have $\rho,\sigma\in \fRep(\hat E,\rr)$ with the same representation of the form (\ref{finrep11}), however they are not isomorphic in the category $\fRep(\hat E,\rr)$, as there is no invertible $2\times 2$ matrix $A\in \M_2(\K)$ such that  $A \rho(e) =\sigma(e) A$.

On the other hand, consider the graph $E$ as above (only one loop) and the graph $F$ consisting of a single isolated vertex. It is easy to observe that 
$\Rep(\hat F,\rr)\cong (\vc \K, \varnothing)$ and $\fRep(\hat F,\rr)\cong (\fvc \K, \varnothing)$ (as there is no relations on the graph). On the other hand $\Rep(\hat E,\rr)\cong (\vc \K,iso)$ and 
$\fRep(\hat E,\rr)\cong (\fvc \K,iso)$ where $iso$ consists of all (appropriate) bijective linear transformations. Therefore the categories $\Rep(\hat E,\rr)$ and $\Rep(\hat F,\rr)$ (and similarly the categories $\fRep(\hat E,\rr)$ and $\fRep(\hat F,\rr)$) are not equivalent. This shows that the categories $\Modd L_\K(E)$ and $\Modd L_\K(F)$ are not equivalent (similarly $\fModd L_\K(E) \not \cong \fModd L_\K(F)$), although their corresponding graph monoids are isomorphic, i.e.,  $M_E\cong M_F$.  
\end{example}

Our goal is to describe locally finite dimensional graded representations of Leavitt path algebras. We first describe the finite dimensional graded representations similar to Proposition~\ref{firstmanin}. The proposition shows the assumption of finite condition on the graded representations is very restrictive. 

\begin{prop}\label{firstmanin2}
Let $E$ be a finite graph such that the sources are isolated vertices.  Then $\rho:\overline{\hat E} \rightarrow \fvc \K$  in  $\fRep(\overline{\hat E},  \overline{\rr})$ has the form 
\begin{equation}\label{finrep}
\rho(v_i)=
 \left\{ \begin{array}{ll}
  \K^{n_{v_i}} & \textrm{if } v  \textrm{ is an isolated vertex},  \\
 0 & \textrm{otherwise,}
  \end{array} \right.
\end{equation}
where all but finite number of $n_{v_i}\in \mathbb N$ is non-zero.
\end{prop}
\begin{proof}
Let $\rho \in \fRep(\overline{\hat E},  \overline{\rr})$.  Since $\dim_\K(\bigoplus_{u\in E^0, i\in \mathbb Z}  \rho(u_i)) < \infty$, there is an $l\in \mathbb Z$ such that $\dim_\K\rho(u_k)=0$ for $k<l$. Choose $i$ to be the smallest integer such that  $\dim_\K (\rho(u_i))\not = 0$ among all non-isolated vertices $u$. 
Since $u$ is not an isolated vertex (thus not a source), there is an edge $e$ with $s(e)=w$ and $r(e)=u$. Note that since $\overline{\rr}$ are Cuntz-Krieger's relations, we have the equations~(\ref{mainprop}) in  $\overline {\hat E}$. Thus $d(w_{i-1})\geq d(u_i)$, which follows that $\dim_\K (\rho(w_{i-1}))\not = 0$. But this is a contradiction. The proposition now follows. 
\end{proof}

Recall that the full subcategory of  $\mathbb{Z}$-locally finite representations of $E$, denoted by  $\ffRep(\overline E,\overline r)$, consists of those functors $\rho  \in \Rep(\overline E,\overline r)$ such that $\dim_\K(\bigoplus_{u\in E^0}  \rho(u_i)) < \infty$, for every $i\in \mathbb Z$ which corresponds to  $\fGr[\mathbb{Z}] A_\K(E,r)$, the category of  locally finite graded representation of the algebra $A_\K(E,r)$ (Proposition~\ref{mainlemma1}). Specialising to the case of Leavitt path algebras, we look into the category $\fGr[\mathbb{Z}] A_\K(\hat{E},r_{ck})=\fGr[\mathbb{Z}]L_{\K}(E)$ by considering the category of  locally finite representations over the covering graph of the double graph $\ffRep(\overline {\hat{E}},\overline{r_{ck}})$.

The dimension function $d$ (see Notations~\ref{notenan}) can be viewed as a \emph{vector space distribution} of the representation. Note again that differing in the morphisms, it is possible for two distinct representations to have the same dimension functions as seen in Example \ref{whenmy}.


\begin{deff}
    A mapping $ d:\{ u_i\}_{i\in \mathbb{Z}}\rightarrow \mathbb{N}$ is said to be eventually trivial if there exists $i\in \mathbb{Z}$ such that for all $j\geq i$, $d(u_j)=0$. 
\end{deff}

In preparation for the main Theorem~\ref{mainthfer}, we need to establish some more statements.

\begin{lemma}
 \label{lemma6}
Let $E$ be a finite graph with no sources and $v\in E^0$. Then there exists a cycle $C$ such that $C^0\geq v$.  
\end{lemma}

\begin{proof}
    Let $E$ be a finite graph with no sources and $v\in E^0$. Then there exists an edge $e_1$ with $r(e_1)=v$. Since $E$ has no sources, there exists an edge $e_2$ with $r(e_2)=s(e_1)$. Continuing, we can obtain edges $e_1,e_2,\cdots$ such that $r(e_i)=s(e_{i-1})$ for each $i=2,\cdots$ with $r(e_1)=v$. Since $E$ is finite and $E$ has no sources, there exists $n$ with $s(e_n)=r(e_j)$ for some $j=1, 2,\cdots,n$. Hence, $C=e_ne_{n-1}\cdots e_j$ is a cycle and $C$ is connected to $v$.
\end{proof}
Recall that for a graph $E$, a cycle $C$ is said to be maximal if no other cycle connects to
$C$ (in particular, a maximal cycle is disjoint from all other cycles). Hence, by a non-maximal cycle, we mean a cycle $C$ where there is another cycle $D$ and a path from $D$ to $C$. The set of vertices in $E$ lying in some maximal cycle is denoted by $C(E)$, that is,   $C(E)=\{v\in E^0 \mid v\in C^0, \textnormal{~for~some~maximal~cycle~} C\} $.

\begin{thm}
    \label{theorem7b}
Let $E$ be a finite graph with no sources and $\rho$ a representation in $\ffRep(\overline{\hat{E}}, \overline{r_{ck}})$. Then for every $v\not \in C(E)$, 
\begin{align*}
    d_v:{\{v_i \}_{i\in \mathbb{Z}}}&\longrightarrow \mathbb{N}\\
    v_i &\longmapsto \dim_\K(\rho(v_i)).
\end{align*}
is eventually trivial. 
\end{thm}

\begin{proof}
    Let $v\not \in C(E)$ and suppose $d_v$ is not eventually trivial. Then for every $i\in \mathbb{Z}$, there exists $j\geq i$ such that $d_v(v_j)\neq 0$.

By Lemma \ref{lemma6}, there exists a maximal cycle $C$, say of length $n$, and a path $p$ with $s(p)\in C^0$ and $r(p)=v$. Since $v\not \in C(E)$,  $v\not \in C^0$. Let $w\in C^0$ such that there exists a path $q$ with $s(q)=w$ and $r(q)=v$ with
$$s:=|q|=\min \{|p| \mid p\in \textnormal{Path}
(E), \textnormal{~with~}s(p)\in C^0, r(p)=v\}.$$
Thus in $\overline{\hat{E}}$, for each $i\in\mathbb{Z}$,  there exists a path of length $s$ from $w_i$ to $v_{i+s}$ and a path of length $n$ from $w_i$ to $w_{i+n}$. Applying Corollary \ref{corollarydimension},  it follows that for each $i$, $d(v_{i+s})$ and $d(w_{i+n})$ are summands in one integer partition of $d(w_i)$. So 
$$d(w_i)\geq d(w_{i+n})\geq d(w_{i+2n})\geq \cdots,$$ and since $v\not \in C^0$, if $d(v_{i+s})\neq 0$,  then $d(w_i)>d(w_{i+n})$.

Now, we partition $\{v_i\}_{i\in \mathbb{Z}}$ into sets $V_i=\{v_{kn+i+s} \mid k\in \mathbb{Z}   \}$, $0\leq i\leq n-1$, and consider the restriction $d_{V_i}:=d_i$ of the dimension function $d:{\overline{\hat{E}}}^0\rightarrow \mathbb{N}$ to $V_i$.
If $d_i$ is eventually trivial for each $0\leq i\leq n-1$, then $d_v$ is eventually trivial, a contradiction. Hence, there exists $0\leq i\leq n-1$
such that $d_i$ is not eventually trivial. Thus, for all $k\in \mathbb{Z}$, there exists $k'>k$ such that $d_i( v_{k'n+i+s} )\neq 0$. We collect all such integer $k'$ and let $$S=\{k_j\mid d(v_{k_jn+i+s}) \neq 0, j=1,2,\dots\}.$$
Since $d_i$ is not eventually trivial, it follows that $S$ is infinite. Without loss of generality, set $k_j<k_{j+1}$ for all $j$. Then 
$$d(w_{k_jn+i})>d(w_{k_{j+1}n+i})> d(w_{k_{j+2}n+i})>\cdots.$$

 As dimensions take non-negative values, we get a contradiction.
\end{proof}

\begin{cor}
    \label{corollaryy2} Let $E$ be a finite graph with no sources and sinks, and $\rho$ a representation in $\ffRep(\overline{\hat{E}}, \overline{r_{ck}})$. Then for $v\not \in C(E)$, we have  $\rho(v_i)=0$, for all $i\in \mathbb Z$.
\end{cor}
\begin{proof}
     Let $v\not \in C(E)$, i.e., $v$ is not a vertex on a maximal cycle. Then by Theorem \ref{theorem7b}, $d_v:\{  v_i\}_{i\in \mathbb{Z}}\rightarrow \mathbb{N} $ is eventually trivial. Thus, there exists $j_{v}\in \mathbb{Z}$ such that $d_v(v_i)=0$, for all $i\geq j_{v}$. Since $E$ has no sinks, $r(s^{-1}(v))\neq \varnothing$. Notice that $r(s^{-1}(v))\cap C(E)=\varnothing$. Since $E$ has finite number of vertices, we can take $t:= \max \{j_{v}\in \mathbb{Z}:v\not \in C(E) \}$. Hence, for $v\not \in C(E)$, $d(u_t)=0$, for all $u\in r(s^{-1}(v))$.  Accordingly, we have
$$d(v_{t-1})=\sum_{u\in r(s^{-1}(v))}d(u_t)=0.$$
 Hence, for $v\not \in C(E)$, $d(v_{t-1})=0$. 
 Similar arguments would give us $d(v_{t-2})=0$. Hence, for all $v\not \in C(E)$, $d(v_i)=0$, i.e., $\phi(v_i)=0$, for all $i\in \mathbb Z$. 
\end{proof}

Recall the notation $\textnormal{CSP}(v)$,  the collection of all simple closed paths based on $v$, and $\text{CSP}(E)_{>1}$, from Section~\ref{graphsec}. Corollary \ref{proposition1-1}  is a direct consequence of Theorem \ref{theorem7b}. 
\begin{cor}
\label{proposition1-1}
 Let $E$ be a finite graph and  $\rho$ a representation in $\ffRep(\overline{\hat{E}}, \overline{r_{ck}})$. Then for every $v\in \textnormal{CSP}(E)_{>1}$, $d_v$ is eventually trivial .
\end{cor}





\begin{thm}
    \label{theoremy1}
Let $E$ be a finite graph with no sources, $C$ a maximal cycle of length $n$ and $\rho$ a representation in $\ffRep(\overline{\hat{E}}, \overline{r_{ck}})$. 
Then there exists $t\in \mathbb{Z}$ such that for each $v_i\in C^0$ $(0\leq i<n)$, there exists $k_i \in \mathbb{N}$ with 
$$d((v_j)_s)=k_i, \text{ where } i\equiv j+t-s \pmod{n},$$
for any $s\geq t$.
\end{thm}

\begin{proof}
    Let $C=e_0e_1\cdots e_{n-1}$ be a maximal cycle of length $n$ in the finite graph $E$,  where $C^0=\{v_0, v_1, \dots, v_{n-1}\}$, with $s(e_{i})=v_{i}=r(e_{i-1})$, where $0< i < n-1$ and $s(e_{n-1})=v_{n-1}$ and $r(e_{n-1})=s(e_0)=v_0$.


Let $w\in E^0\setminus C^0$. If $w\in C(E)$, then there is no path from $C$ to $w$ since $w$ is in a maximal cycle and $C$ is also a maximal cycle. Thus, $d(w_i)$ does not contribute to $d((v_m)_j)$, for all $i,j\in \mathbb{Z}$ and $m=0,1,\cdots , n-1$. If $w\not \in C(E)$, then by Theorem \ref{theorem7b}, $d_w $ is eventually trivial, that is, there exists $j_w$ such that for all $i\geq j_w$, $d(w_i)=0$. Let $t=\max\{ j_w \mid w\in E^0\setminus C(E)\}$. Then for all $w\in E^0\setminus C(E)$ and $i\geq t$, $d(w_i)=0$. 



Let $v_m\in C^0$, $m=0,1,\cdots, n-1$ and $j \geq t$. Let $w\in r(s^{-1}(v_m))$. If $w\not \in C^0$, then $w\not \in C(E)$ and thus $d(w_j)=0$. Thus by Corollary~\ref{corollarydimension}, 
$$d((v_m)_j)= \displaystyle\sum_{e\in s^{-1}(v_m)} d(r(e)_{j+1})=  d((v_{m'})_{j+1}) $$
where $m'=m+1\pmod{n}$. 

Hence, for each $i=0,1,\cdots, n-1$, every vertex in 
$$p_i=(e_i)_t(e_{i+1})_{t+1}(e_{i+2})_{t+2}\cdots (e_{n-1})_{t+n-1-i}(e_0)_{t+n-i}(e_1)_{t+n-i+1}\cdots $$ has equal dimension, say $k_i$. 
Let $s\geq t $ and $(v_j)_s \in p_i^0$, that is, $d((v_j)_s)=k_i=d((v_i)_t)$.  
Notice that $s-j=t-i+nq$, for some $q\in \mathbb{N}$. 
Hence, $d((v_j)_s)=k_i$, whenever    $i \equiv j+t-s \pmod{n}$.
\end{proof}

\begin{example}
    Consider the following graph $E$ and let $\rho\in \ffRep(\overline{\hat{E}}, \overline{r_{ck}})$ with the dimension function $d^\rho=d$ as illustrated below. Here, we replace the vertices with the dimension in the representation and
    for simplicity, we omit the ghost edges in the graph $\overline{\hat{E}}$. The dimension $d(\alpha_i)$ of a vertex $\alpha_i$,  is in row $\alpha$ and column $i$. For instance, $d((v_3)_2)=2$.

 By Corollary \ref{corollarydimension}, it could be verified that indeed, we obtain a vector space distribution for a finite representation. 
    \begin{center}

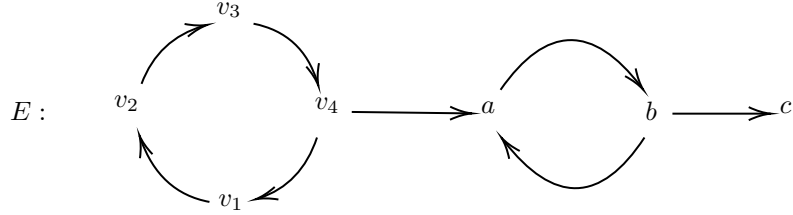
\tikzset{every picture/.style={line width=0.75pt}} 

\begin{tikzpicture}[x=0.75pt,y=0.75pt,yscale=-1,xscale=1]

\draw (114,119.4) node [anchor=north west][inner sep=0.75pt]    {$v_{1}$};
\draw (62,69.4) node [anchor=north west][inner sep=0.75pt]    {$v_{2}$};
\draw (113,22.4) node [anchor=north west][inner sep=0.75pt]    {$v_{3}$};
\draw (162,70.4) node [anchor=north west][inner sep=0.75pt]    {$v_{4}$};
\draw (245,72.4) node [anchor=north west][inner sep=0.75pt]    {$a$};
\draw (327,72.4) node [anchor=north west][inner sep=0.75pt]    {$b$};
\draw (394,72.4) node [anchor=north west][inner sep=0.75pt]    {$c$};
\draw (10,72.4) node [anchor=north west][inner sep=0.75pt]    {$E:$};
\draw    (77.12,65) .. controls (82.71,50) and (93.03,40.28) .. (108.09,35.85) ;
\draw [shift={(110,35.32)}, rotate = 165.52] [color={rgb, 255:red, 0; green, 0; blue, 0 }  ][line width=0.75]    (10.93,-3.29) .. controls (6.95,-1.4) and (3.31,-0.3) .. (0,0) .. controls (3.31,0.3) and (6.95,1.4) .. (10.93,3.29)   ;
\draw    (133,34.75) .. controls (149.23,37.99) and (159.84,47.81) .. (164.83,64.19) ;
\draw [shift={(165.35,66)}, rotate = 254.78] [color={rgb, 255:red, 0; green, 0; blue, 0 }  ][line width=0.75]    (10.93,-3.29) .. controls (6.95,-1.4) and (3.31,-0.3) .. (0,0) .. controls (3.31,0.3) and (6.95,1.4) .. (10.93,3.29)   ;
\draw    (164.65,92) .. controls (159.46,107.71) and (149.83,117.88) .. (135.78,122.53) ;
\draw [shift={(134,123.08)}, rotate = 343.8] [color={rgb, 255:red, 0; green, 0; blue, 0 }  ][line width=0.75]    (10.93,-3.29) .. controls (6.95,-1.4) and (3.31,-0.3) .. (0,0) .. controls (3.31,0.3) and (6.95,1.4) .. (10.93,3.29)   ;
\draw    (111,124.4) .. controls (94.41,121.33) and (82.96,110.75) .. (76.66,92.7) ;
\draw [shift={(76.09,91)}, rotate = 72.09] [color={rgb, 255:red, 0; green, 0; blue, 0 }  ][line width=0.75]    (10.93,-3.29) .. controls (6.95,-1.4) and (3.31,-0.3) .. (0,0) .. controls (3.31,0.3) and (6.95,1.4) .. (10.93,3.29)   ;
\draw    (182,79.14) -- (240,79.86) ;
\draw [shift={(242,79.89)}, rotate = 180.71] [color={rgb, 255:red, 0; green, 0; blue, 0 }  ][line width=0.75]    (10.93,-3.29) .. controls (6.95,-1.4) and (3.31,-0.3) .. (0,0) .. controls (3.31,0.3) and (6.95,1.4) .. (10.93,3.29)   ;
\draw    (256.25,68) .. controls (278.96,35.17) and (302.34,34.68) .. (326.37,66.52) ;
\draw [shift={(327.47,68)}, rotate = 233.77] [color={rgb, 255:red, 0; green, 0; blue, 0 }  ][line width=0.75]    (10.93,-3.29) .. controls (6.95,-1.4) and (3.31,-0.3) .. (0,0) .. controls (3.31,0.3) and (6.95,1.4) .. (10.93,3.29)   ;
\draw    (328.05,92) .. controls (306.23,125.49) and (282.81,125.99) .. (257.78,93.51) ;
\draw [shift={(256.63,92)}, rotate = 53.16] [color={rgb, 255:red, 0; green, 0; blue, 0 }  ][line width=0.75]    (10.93,-3.29) .. controls (6.95,-1.4) and (3.31,-0.3) .. (0,0) .. controls (3.31,0.3) and (6.95,1.4) .. (10.93,3.29)   ;
\draw    (342,80) -- (389,80) ;
\draw [shift={(391,80)}, rotate = 180] [color={rgb, 255:red, 0; green, 0; blue, 0 }  ][line width=0.75]    (10.93,-3.29) .. controls (6.95,-1.4) and (3.31,-0.3) .. (0,0) .. controls (3.31,0.3) and (6.95,1.4) .. (10.93,3.29)   ;

\end{tikzpicture}
    
    \end{center}
    ~\\
    \begin{center}

\tikzset{every picture/.style={line width=0.75pt}} 

\tikzset{every picture/.style={line width=0.75pt}} 

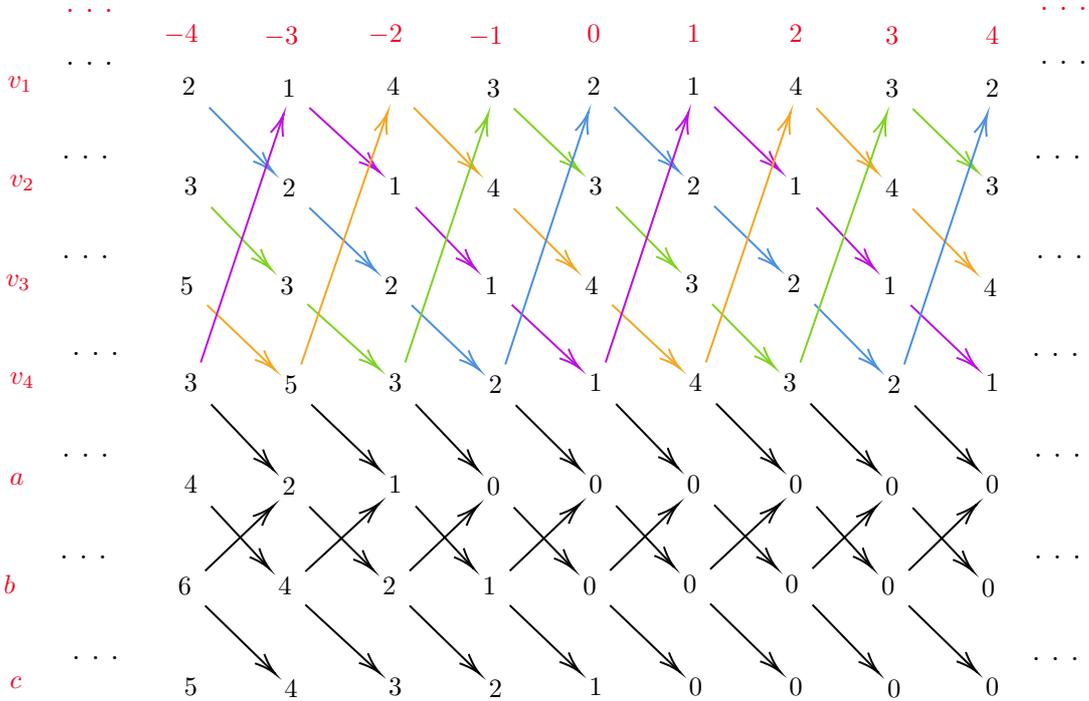
\begin{tikzpicture}[x=0.75pt,y=0.75pt,yscale=-1,xscale=1]

\draw (129,56.4) node [anchor=north west][inner sep=0.75pt]  [font=\Large]  {$.\ .\ .$};
\draw (127,103.4) node [anchor=north west][inner sep=0.75pt]  [font=\Large]  {$.\ .\ .$};
\draw (127,153.4) node [anchor=north west][inner sep=0.75pt]  [font=\Large]  {$.\ .\ .$};
\draw (132,202.4) node [anchor=north west][inner sep=0.75pt]  [font=\Large]  {$.\ .\ .$};
\draw (615,55.47) node [anchor=north west][inner sep=0.75pt]  [font=\Large]  {$.\ .\ .$};
\draw (612,103.47) node [anchor=north west][inner sep=0.75pt]  [font=\Large]  {$.\ .\ .$};
\draw (613,153.47) node [anchor=north west][inner sep=0.75pt]  [font=\Large]  {$.\ .\ .$};
\draw (611,202.47) node [anchor=north west][inner sep=0.75pt]  [font=\Large]  {$.\ .\ .$};
\draw (439,64.4) node [anchor=north west][inner sep=0.75pt]    {$1$};
\draw (389,64.4) node [anchor=north west][inner sep=0.75pt]    {$2$};
\draw (490,64.4) node [anchor=north west][inner sep=0.75pt]    {$4$};
\draw (538,65.4) node [anchor=north west][inner sep=0.75pt]    {$3$};
\draw (588,65.4) node [anchor=north west][inner sep=0.75pt]    {$2$};
\draw (439,114.4) node [anchor=north west][inner sep=0.75pt]    {$2$};
\draw (390,114.4) node [anchor=north west][inner sep=0.75pt]    {$3$};
\draw (490,114.4) node [anchor=north west][inner sep=0.75pt]    {$1$};
\draw (538,115.4) node [anchor=north west][inner sep=0.75pt]    {$4$};
\draw (588,114.4) node [anchor=north west][inner sep=0.75pt]    {$3$};
\draw (438,163.4) node [anchor=north west][inner sep=0.75pt]    {$3$};
\draw (388,164.4) node [anchor=north west][inner sep=0.75pt]    {$4$};
\draw (489,163.4) node [anchor=north west][inner sep=0.75pt]    {$2$};
\draw (537,164.4) node [anchor=north west][inner sep=0.75pt]    {$1$};
\draw (587,165.4) node [anchor=north west][inner sep=0.75pt]    {$4$};
\draw (440,213.4) node [anchor=north west][inner sep=0.75pt]    {$4$};
\draw (390,213.4) node [anchor=north west][inner sep=0.75pt]    {$1$};
\draw (487,213.4) node [anchor=north west][inner sep=0.75pt]    {$3$};
\draw (539,214.4) node [anchor=north west][inner sep=0.75pt]    {$2$};
\draw (588,213.4) node [anchor=north west][inner sep=0.75pt]    {$1$};
\draw (339,65.4) node [anchor=north west][inner sep=0.75pt]    {$3$};
\draw (289,64.4) node [anchor=north west][inner sep=0.75pt]    {$4$};
\draw (339,115.4) node [anchor=north west][inner sep=0.75pt]    {$4$};
\draw (290,114.4) node [anchor=north west][inner sep=0.75pt]    {$1$};
\draw (338,164.4) node [anchor=north west][inner sep=0.75pt]    {$1$};
\draw (288,164.4) node [anchor=north west][inner sep=0.75pt]    {$2$};
\draw (340,214.4) node [anchor=north west][inner sep=0.75pt]    {$2$};
\draw (290,213.4) node [anchor=north west][inner sep=0.75pt]    {$3$};
\draw (127,253.4) node [anchor=north west][inner sep=0.75pt]  [font=\Large]  {$.\ .\ .$};
\draw (126,304.4) node [anchor=north west][inner sep=0.75pt]  [font=\Large]  {$.\ .\ .$};
\draw (132,355.4) node [anchor=north west][inner sep=0.75pt]  [font=\Large]  {$.\ .\ .$};
\draw (612,253.47) node [anchor=north west][inner sep=0.75pt]  [font=\Large]  {$.\ .\ .$};
\draw (612,304.47) node [anchor=north west][inner sep=0.75pt]  [font=\Large]  {$.\ .\ .$};
\draw (611,355.47) node [anchor=north west][inner sep=0.75pt]  [font=\Large]  {$.\ .\ .$};
\draw (439,264.4) node [anchor=north west][inner sep=0.75pt]    {$0$};
\draw (390,264.4) node [anchor=north west][inner sep=0.75pt]    {$0$};
\draw (490,264.4) node [anchor=north west][inner sep=0.75pt]    {$0$};
\draw (538,265.4) node [anchor=north west][inner sep=0.75pt]    {$0$};
\draw (588,264.4) node [anchor=north west][inner sep=0.75pt]    {$0$};
\draw (437,314.4) node [anchor=north west][inner sep=0.75pt]    {$0$};
\draw (387,315.4) node [anchor=north west][inner sep=0.75pt]    {$0$};
\draw (488,314.4) node [anchor=north west][inner sep=0.75pt]    {$0$};
\draw (536,315.4) node [anchor=north west][inner sep=0.75pt]    {$0$};
\draw (586,316.4) node [anchor=north west][inner sep=0.75pt]    {$0$};
\draw (440,366.4) node [anchor=north west][inner sep=0.75pt]    {$0$};
\draw (390,366.4) node [anchor=north west][inner sep=0.75pt]    {$1$};
\draw (490,366.4) node [anchor=north west][inner sep=0.75pt]    {$0$};
\draw (539,367.4) node [anchor=north west][inner sep=0.75pt]    {$0$};
\draw (588,366.4) node [anchor=north west][inner sep=0.75pt]    {$0$};
\draw (339,265.4) node [anchor=north west][inner sep=0.75pt]    {$0$};
\draw (290,264.4) node [anchor=north west][inner sep=0.75pt]    {$1$};
\draw (337,315.4) node [anchor=north west][inner sep=0.75pt]    {$1$};
\draw (287,315.4) node [anchor=north west][inner sep=0.75pt]    {$2$};
\draw (340,367.4) node [anchor=north west][inner sep=0.75pt]    {$2$};
\draw (290,366.4) node [anchor=north west][inner sep=0.75pt]    {$3$};
\draw (100,63.4) node [anchor=north west][inner sep=0.75pt]  [color={rgb, 255:red, 237; green, 7; blue, 34 }  ,opacity=1 ]  {$v_{1}$};
\draw (101,113.4) node [anchor=north west][inner sep=0.75pt]  [color={rgb, 255:red, 237; green, 7; blue, 34 }  ,opacity=1 ]  {$v_{2}$};
\draw (99,163.4) node [anchor=north west][inner sep=0.75pt]  [color={rgb, 255:red, 237; green, 7; blue, 34 }  ,opacity=1 ]  {$v_{3}$};
\draw (101,212.4) node [anchor=north west][inner sep=0.75pt]  [color={rgb, 255:red, 237; green, 7; blue, 34 }  ,opacity=1 ]  {$v_{4}$};
\draw (101,263.4) node [anchor=north west][inner sep=0.75pt]  [color={rgb, 255:red, 237; green, 7; blue, 34 }  ,opacity=1 ]  {$a$};
\draw (98,314.4) node [anchor=north west][inner sep=0.75pt]  [color={rgb, 255:red, 237; green, 7; blue, 34 }  ,opacity=1 ]  {$b$};
\draw (101,365.4) node [anchor=north west][inner sep=0.75pt]  [color={rgb, 255:red, 237; green, 7; blue, 34 }  ,opacity=1 ]  {$c$};
\draw (129,29.4) node [anchor=north west][inner sep=0.75pt]  [font=\Large,color={rgb, 255:red, 237; green, 7; blue, 34 }  ,opacity=1 ]  {$.\ .\ .$};
\draw (615,28.47) node [anchor=north west][inner sep=0.75pt]  [font=\Large,color={rgb, 255:red, 237; green, 7; blue, 34 }  ,opacity=1 ]  {$.\ .\ .$};
\draw (439,37.4) node [anchor=north west][inner sep=0.75pt]  [color={rgb, 255:red, 237; green, 7; blue, 34 }  ,opacity=1 ]  {$1$};
\draw (389,37.4) node [anchor=north west][inner sep=0.75pt]  [color={rgb, 255:red, 237; green, 7; blue, 34 }  ,opacity=1 ]  {$0$};
\draw (490,37.4) node [anchor=north west][inner sep=0.75pt]  [color={rgb, 255:red, 237; green, 7; blue, 34 }  ,opacity=1 ]  {$2$};
\draw (538,38.4) node [anchor=north west][inner sep=0.75pt]  [color={rgb, 255:red, 237; green, 7; blue, 34 }  ,opacity=1 ]  {$3$};
\draw (588,38.4) node [anchor=north west][inner sep=0.75pt]  [color={rgb, 255:red, 237; green, 7; blue, 34 }  ,opacity=1 ]  {$4$};
\draw (330,38.4) node [anchor=north west][inner sep=0.75pt]  [color={rgb, 255:red, 237; green, 7; blue, 34 }  ,opacity=1 ]  {$-1$};
\draw (280,37.4) node [anchor=north west][inner sep=0.75pt]  [color={rgb, 255:red, 237; green, 7; blue, 34 }  ,opacity=1 ]  {$-2$};
\draw (237,65.4) node [anchor=north west][inner sep=0.75pt]    {$1$};
\draw (187,64.4) node [anchor=north west][inner sep=0.75pt]    {$2$};
\draw (237,115.4) node [anchor=north west][inner sep=0.75pt]    {$2$};
\draw (188,114.4) node [anchor=north west][inner sep=0.75pt]    {$3$};
\draw (236,164.4) node [anchor=north west][inner sep=0.75pt]    {$3$};
\draw (186,164.4) node [anchor=north west][inner sep=0.75pt]    {$5$};
\draw (238,214.4) node [anchor=north west][inner sep=0.75pt]    {$5$};
\draw (188,213.4) node [anchor=north west][inner sep=0.75pt]    {$3$};
\draw (237,265.4) node [anchor=north west][inner sep=0.75pt]    {$2$};
\draw (188,264.4) node [anchor=north west][inner sep=0.75pt]    {$4$};
\draw (235,315.4) node [anchor=north west][inner sep=0.75pt]    {$4$};
\draw (185,315.4) node [anchor=north west][inner sep=0.75pt]    {$6$};
\draw (238,367.4) node [anchor=north west][inner sep=0.75pt]    {$4$};
\draw (188,366.4) node [anchor=north west][inner sep=0.75pt]    {$5$};
\draw (228,38.4) node [anchor=north west][inner sep=0.75pt]  [color={rgb, 255:red, 237; green, 7; blue, 34 }  ,opacity=1 ]  {$-3$};
\draw (178,37.4) node [anchor=north west][inner sep=0.75pt]  [color={rgb, 255:red, 237; green, 7; blue, 34 }  ,opacity=1 ]  {$-4$};
\draw [color={rgb, 255:red, 245; green, 166; blue, 35 }  ,draw opacity=1 ]   (304,80.68) -- (334.6,111.89) ;
\draw [shift={(336,113.32)}, rotate = 225.57] [color={rgb, 255:red, 245; green, 166; blue, 35 }  ,draw opacity=1 ][line width=0.75]    (10.93,-3.29) .. controls (6.95,-1.4) and (3.31,-0.3) .. (0,0) .. controls (3.31,0.3) and (6.95,1.4) .. (10.93,3.29)   ;
\draw [color={rgb, 255:red, 189; green, 16; blue, 224 }  ,draw opacity=1 ]   (305,130.88) -- (333.61,160.68) ;
\draw [shift={(335,162.13)}, rotate = 226.17] [color={rgb, 255:red, 189; green, 16; blue, 224 }  ,draw opacity=1 ][line width=0.75]    (10.93,-3.29) .. controls (6.95,-1.4) and (3.31,-0.3) .. (0,0) .. controls (3.31,0.3) and (6.95,1.4) .. (10.93,3.29)   ;
\draw [color={rgb, 255:red, 74; green, 144; blue, 226 }  ,draw opacity=1 ]   (303,180.15) -- (335.56,211.46) ;
\draw [shift={(337,212.85)}, rotate = 223.88] [color={rgb, 255:red, 74; green, 144; blue, 226 }  ,draw opacity=1 ][line width=0.75]    (10.93,-3.29) .. controls (6.95,-1.4) and (3.31,-0.3) .. (0,0) .. controls (3.31,0.3) and (6.95,1.4) .. (10.93,3.29)   ;
\draw    (305,230.05) -- (334.63,261.49) ;
\draw [shift={(336,262.95)}, rotate = 226.7] [color={rgb, 255:red, 0; green, 0; blue, 0 }  ][line width=0.75]    (10.93,-3.29) .. controls (6.95,-1.4) and (3.31,-0.3) .. (0,0) .. controls (3.31,0.3) and (6.95,1.4) .. (10.93,3.29)   ;
\draw    (305,281.27) -- (332.64,311.26) ;
\draw [shift={(334,312.73)}, rotate = 227.34] [color={rgb, 255:red, 0; green, 0; blue, 0 }  ][line width=0.75]    (10.93,-3.29) .. controls (6.95,-1.4) and (3.31,-0.3) .. (0,0) .. controls (3.31,0.3) and (6.95,1.4) .. (10.93,3.29)   ;
\draw    (302,331.33) -- (335.57,364.27) ;
\draw [shift={(337,365.67)}, rotate = 224.45] [color={rgb, 255:red, 0; green, 0; blue, 0 }  ][line width=0.75]    (10.93,-3.29) .. controls (6.95,-1.4) and (3.31,-0.3) .. (0,0) .. controls (3.31,0.3) and (6.95,1.4) .. (10.93,3.29)   ;
\draw [color={rgb, 255:red, 126; green, 211; blue, 33 }  ,draw opacity=1 ]   (299.81,209) -- (340.56,85.9) ;
\draw [shift={(341.19,84)}, rotate = 108.32] [color={rgb, 255:red, 126; green, 211; blue, 33 }  ,draw opacity=1 ][line width=0.75]    (10.93,-3.29) .. controls (6.95,-1.4) and (3.31,-0.3) .. (0,0) .. controls (3.31,0.3) and (6.95,1.4) .. (10.93,3.29)   ;
\draw [color={rgb, 255:red, 126; green, 211; blue, 33 }  ,draw opacity=1 ]   (354,81.15) -- (385.56,111.47) ;
\draw [shift={(387,112.85)}, rotate = 223.85] [color={rgb, 255:red, 126; green, 211; blue, 33 }  ,draw opacity=1 ][line width=0.75]    (10.93,-3.29) .. controls (6.95,-1.4) and (3.31,-0.3) .. (0,0) .. controls (3.31,0.3) and (6.95,1.4) .. (10.93,3.29)   ;
\draw [color={rgb, 255:red, 245; green, 166; blue, 35 }  ,draw opacity=1 ]   (354,131.5) -- (383.59,161.09) ;
\draw [shift={(385,162.5)}, rotate = 225] [color={rgb, 255:red, 245; green, 166; blue, 35 }  ,draw opacity=1 ][line width=0.75]    (10.93,-3.29) .. controls (6.95,-1.4) and (3.31,-0.3) .. (0,0) .. controls (3.31,0.3) and (6.95,1.4) .. (10.93,3.29)   ;
\draw [color={rgb, 255:red, 189; green, 16; blue, 224 }  ,draw opacity=1 ]   (353,179.98) -- (385.54,210.65) ;
\draw [shift={(387,212.02)}, rotate = 223.3] [color={rgb, 255:red, 189; green, 16; blue, 224 }  ,draw opacity=1 ][line width=0.75]    (10.93,-3.29) .. controls (6.95,-1.4) and (3.31,-0.3) .. (0,0) .. controls (3.31,0.3) and (6.95,1.4) .. (10.93,3.29)   ;
\draw    (354,281.88) -- (382.61,311.68) ;
\draw [shift={(384,313.13)}, rotate = 226.17] [color={rgb, 255:red, 0; green, 0; blue, 0 }  ][line width=0.75]    (10.93,-3.29) .. controls (6.95,-1.4) and (3.31,-0.3) .. (0,0) .. controls (3.31,0.3) and (6.95,1.4) .. (10.93,3.29)   ;
\draw    (352,331.16) -- (385.56,363.45) ;
\draw [shift={(387,364.84)}, rotate = 223.9] [color={rgb, 255:red, 0; green, 0; blue, 0 }  ][line width=0.75]    (10.93,-3.29) .. controls (6.95,-1.4) and (3.31,-0.3) .. (0,0) .. controls (3.31,0.3) and (6.95,1.4) .. (10.93,3.29)   ;
\draw [color={rgb, 255:red, 74; green, 144; blue, 226 }  ,draw opacity=1 ]   (349.76,210) -- (390.62,84.9) ;
\draw [shift={(391.24,83)}, rotate = 108.09] [color={rgb, 255:red, 74; green, 144; blue, 226 }  ,draw opacity=1 ][line width=0.75]    (10.93,-3.29) .. controls (6.95,-1.4) and (3.31,-0.3) .. (0,0) .. controls (3.31,0.3) and (6.95,1.4) .. (10.93,3.29)   ;
\draw [color={rgb, 255:red, 74; green, 144; blue, 226 }  ,draw opacity=1 ]   (404,80.5) -- (434.59,111.09) ;
\draw [shift={(436,112.5)}, rotate = 225] [color={rgb, 255:red, 74; green, 144; blue, 226 }  ,draw opacity=1 ][line width=0.75]    (10.93,-3.29) .. controls (6.95,-1.4) and (3.31,-0.3) .. (0,0) .. controls (3.31,0.3) and (6.95,1.4) .. (10.93,3.29)   ;
\draw [color={rgb, 255:red, 126; green, 211; blue, 33 }  ,draw opacity=1 ]   (405,130.69) -- (433.6,159.88) ;
\draw [shift={(435,161.31)}, rotate = 225.59] [color={rgb, 255:red, 126; green, 211; blue, 33 }  ,draw opacity=1 ][line width=0.75]    (10.93,-3.29) .. controls (6.95,-1.4) and (3.31,-0.3) .. (0,0) .. controls (3.31,0.3) and (6.95,1.4) .. (10.93,3.29)   ;
\draw [color={rgb, 255:red, 245; green, 166; blue, 35 }  ,draw opacity=1 ]   (403,179.98) -- (435.54,210.65) ;
\draw [shift={(437,212.02)}, rotate = 223.3] [color={rgb, 255:red, 245; green, 166; blue, 35 }  ,draw opacity=1 ][line width=0.75]    (10.93,-3.29) .. controls (6.95,-1.4) and (3.31,-0.3) .. (0,0) .. controls (3.31,0.3) and (6.95,1.4) .. (10.93,3.29)   ;
\draw    (405,229.87) -- (434.61,260.69) ;
\draw [shift={(436,262.13)}, rotate = 226.15] [color={rgb, 255:red, 0; green, 0; blue, 0 }  ][line width=0.75]    (10.93,-3.29) .. controls (6.95,-1.4) and (3.31,-0.3) .. (0,0) .. controls (3.31,0.3) and (6.95,1.4) .. (10.93,3.29)   ;
\draw    (355,230.5) -- (385.59,261.09) ;
\draw [shift={(387,262.5)}, rotate = 225] [color={rgb, 255:red, 0; green, 0; blue, 0 }  ][line width=0.75]    (10.93,-3.29) .. controls (6.95,-1.4) and (3.31,-0.3) .. (0,0) .. controls (3.31,0.3) and (6.95,1.4) .. (10.93,3.29)   ;
\draw    (405,281.07) -- (432.63,310.47) ;
\draw [shift={(434,311.93)}, rotate = 226.77] [color={rgb, 255:red, 0; green, 0; blue, 0 }  ][line width=0.75]    (10.93,-3.29) .. controls (6.95,-1.4) and (3.31,-0.3) .. (0,0) .. controls (3.31,0.3) and (6.95,1.4) .. (10.93,3.29)   ;
\draw    (402,331.16) -- (435.56,363.45) ;
\draw [shift={(437,364.84)}, rotate = 223.9] [color={rgb, 255:red, 0; green, 0; blue, 0 }  ][line width=0.75]    (10.93,-3.29) .. controls (6.95,-1.4) and (3.31,-0.3) .. (0,0) .. controls (3.31,0.3) and (6.95,1.4) .. (10.93,3.29)   ;
\draw [color={rgb, 255:red, 189; green, 16; blue, 224 }  ,draw opacity=1 ]   (399.78,209) -- (440.59,84.9) ;
\draw [shift={(441.22,83)}, rotate = 108.2] [color={rgb, 255:red, 189; green, 16; blue, 224 }  ,draw opacity=1 ][line width=0.75]    (10.93,-3.29) .. controls (6.95,-1.4) and (3.31,-0.3) .. (0,0) .. controls (3.31,0.3) and (6.95,1.4) .. (10.93,3.29)   ;
\draw [color={rgb, 255:red, 189; green, 16; blue, 224 }  ,draw opacity=1 ]   (454,80.32) -- (485.57,111.28) ;
\draw [shift={(487,112.68)}, rotate = 224.43] [color={rgb, 255:red, 189; green, 16; blue, 224 }  ,draw opacity=1 ][line width=0.75]    (10.93,-3.29) .. controls (6.95,-1.4) and (3.31,-0.3) .. (0,0) .. controls (3.31,0.3) and (6.95,1.4) .. (10.93,3.29)   ;
\draw [color={rgb, 255:red, 74; green, 144; blue, 226 }  ,draw opacity=1 ]   (454,130.32) -- (484.57,160.28) ;
\draw [shift={(486,161.68)}, rotate = 224.42] [color={rgb, 255:red, 74; green, 144; blue, 226 }  ,draw opacity=1 ][line width=0.75]    (10.93,-3.29) .. controls (6.95,-1.4) and (3.31,-0.3) .. (0,0) .. controls (3.31,0.3) and (6.95,1.4) .. (10.93,3.29)   ;
\draw [color={rgb, 255:red, 126; green, 211; blue, 33 }  ,draw opacity=1 ]   (453,179.68) -- (482.6,209.89) ;
\draw [shift={(484,211.32)}, rotate = 225.58] [color={rgb, 255:red, 126; green, 211; blue, 33 }  ,draw opacity=1 ][line width=0.75]    (10.93,-3.29) .. controls (6.95,-1.4) and (3.31,-0.3) .. (0,0) .. controls (3.31,0.3) and (6.95,1.4) .. (10.93,3.29)   ;
\draw    (455,229.68) -- (485.6,260.89) ;
\draw [shift={(487,262.32)}, rotate = 225.57] [color={rgb, 255:red, 0; green, 0; blue, 0 }  ][line width=0.75]    (10.93,-3.29) .. controls (6.95,-1.4) and (3.31,-0.3) .. (0,0) .. controls (3.31,0.3) and (6.95,1.4) .. (10.93,3.29)   ;
\draw    (454,280.68) -- (483.6,310.89) ;
\draw [shift={(485,312.32)}, rotate = 225.58] [color={rgb, 255:red, 0; green, 0; blue, 0 }  ][line width=0.75]    (10.93,-3.29) .. controls (6.95,-1.4) and (3.31,-0.3) .. (0,0) .. controls (3.31,0.3) and (6.95,1.4) .. (10.93,3.29)   ;
\draw    (452,330.33) -- (485.57,363.27) ;
\draw [shift={(487,364.67)}, rotate = 224.45] [color={rgb, 255:red, 0; green, 0; blue, 0 }  ][line width=0.75]    (10.93,-3.29) .. controls (6.95,-1.4) and (3.31,-0.3) .. (0,0) .. controls (3.31,0.3) and (6.95,1.4) .. (10.93,3.29)   ;
\draw [color={rgb, 255:red, 245; green, 166; blue, 35 }  ,draw opacity=1 ]   (449.86,209) -- (491.5,84.9) ;
\draw [shift={(492.14,83)}, rotate = 108.55] [color={rgb, 255:red, 245; green, 166; blue, 35 }  ,draw opacity=1 ][line width=0.75]    (10.93,-3.29) .. controls (6.95,-1.4) and (3.31,-0.3) .. (0,0) .. controls (3.31,0.3) and (6.95,1.4) .. (10.93,3.29)   ;
\draw [color={rgb, 255:red, 245; green, 166; blue, 35 }  ,draw opacity=1 ]   (505,81.06) -- (533.63,111.48) ;
\draw [shift={(535,112.94)}, rotate = 226.74] [color={rgb, 255:red, 245; green, 166; blue, 35 }  ,draw opacity=1 ][line width=0.75]    (10.93,-3.29) .. controls (6.95,-1.4) and (3.31,-0.3) .. (0,0) .. controls (3.31,0.3) and (6.95,1.4) .. (10.93,3.29)   ;
\draw [color={rgb, 255:red, 189; green, 16; blue, 224 }  ,draw opacity=1 ]   (505,131.07) -- (532.63,160.47) ;
\draw [shift={(534,161.93)}, rotate = 226.77] [color={rgb, 255:red, 189; green, 16; blue, 224 }  ,draw opacity=1 ][line width=0.75]    (10.93,-3.29) .. controls (6.95,-1.4) and (3.31,-0.3) .. (0,0) .. controls (3.31,0.3) and (6.95,1.4) .. (10.93,3.29)   ;
\draw [color={rgb, 255:red, 74; green, 144; blue, 226 }  ,draw opacity=1 ]   (504,179.68) -- (534.6,210.89) ;
\draw [shift={(536,212.32)}, rotate = 225.57] [color={rgb, 255:red, 74; green, 144; blue, 226 }  ,draw opacity=1 ][line width=0.75]    (10.93,-3.29) .. controls (6.95,-1.4) and (3.31,-0.3) .. (0,0) .. controls (3.31,0.3) and (6.95,1.4) .. (10.93,3.29)   ;
\draw    (502,229.68) -- (533.6,261.9) ;
\draw [shift={(535,263.32)}, rotate = 225.56] [color={rgb, 255:red, 0; green, 0; blue, 0 }  ][line width=0.75]    (10.93,-3.29) .. controls (6.95,-1.4) and (3.31,-0.3) .. (0,0) .. controls (3.31,0.3) and (6.95,1.4) .. (10.93,3.29)   ;
\draw    (505,281.48) -- (531.66,311.04) ;
\draw [shift={(533,312.52)}, rotate = 227.95] [color={rgb, 255:red, 0; green, 0; blue, 0 }  ][line width=0.75]    (10.93,-3.29) .. controls (6.95,-1.4) and (3.31,-0.3) .. (0,0) .. controls (3.31,0.3) and (6.95,1.4) .. (10.93,3.29)   ;
\draw    (503,330.85) -- (534.61,363.71) ;
\draw [shift={(536,365.15)}, rotate = 226.1] [color={rgb, 255:red, 0; green, 0; blue, 0 }  ][line width=0.75]    (10.93,-3.29) .. controls (6.95,-1.4) and (3.31,-0.3) .. (0,0) .. controls (3.31,0.3) and (6.95,1.4) .. (10.93,3.29)   ;
\draw [color={rgb, 255:red, 126; green, 211; blue, 33 }  ,draw opacity=1 ]   (496.96,209) -- (539.39,85.89) ;
\draw [shift={(540.04,84)}, rotate = 109.01] [color={rgb, 255:red, 126; green, 211; blue, 33 }  ,draw opacity=1 ][line width=0.75]    (10.93,-3.29) .. controls (6.95,-1.4) and (3.31,-0.3) .. (0,0) .. controls (3.31,0.3) and (6.95,1.4) .. (10.93,3.29)   ;
\draw [color={rgb, 255:red, 126; green, 211; blue, 33 }  ,draw opacity=1 ]   (553,81.32) -- (583.57,111.28) ;
\draw [shift={(585,112.68)}, rotate = 224.42] [color={rgb, 255:red, 126; green, 211; blue, 33 }  ,draw opacity=1 ][line width=0.75]    (10.93,-3.29) .. controls (6.95,-1.4) and (3.31,-0.3) .. (0,0) .. controls (3.31,0.3) and (6.95,1.4) .. (10.93,3.29)   ;
\draw [color={rgb, 255:red, 245; green, 166; blue, 35 }  ,draw opacity=1 ]   (553,131.68) -- (582.6,161.89) ;
\draw [shift={(584,163.32)}, rotate = 225.58] [color={rgb, 255:red, 245; green, 166; blue, 35 }  ,draw opacity=1 ][line width=0.75]    (10.93,-3.29) .. controls (6.95,-1.4) and (3.31,-0.3) .. (0,0) .. controls (3.31,0.3) and (6.95,1.4) .. (10.93,3.29)   ;
\draw [color={rgb, 255:red, 189; green, 16; blue, 224 }  ,draw opacity=1 ]   (552,180.15) -- (583.56,210.47) ;
\draw [shift={(585,211.85)}, rotate = 223.85] [color={rgb, 255:red, 189; green, 16; blue, 224 }  ,draw opacity=1 ][line width=0.75]    (10.93,-3.29) .. controls (6.95,-1.4) and (3.31,-0.3) .. (0,0) .. controls (3.31,0.3) and (6.95,1.4) .. (10.93,3.29)   ;
\draw    (554,230.68) -- (583.6,260.89) ;
\draw [shift={(585,262.32)}, rotate = 225.58] [color={rgb, 255:red, 0; green, 0; blue, 0 }  ][line width=0.75]    (10.93,-3.29) .. controls (6.95,-1.4) and (3.31,-0.3) .. (0,0) .. controls (3.31,0.3) and (6.95,1.4) .. (10.93,3.29)   ;
\draw    (553,282.06) -- (581.63,312.48) ;
\draw [shift={(583,313.94)}, rotate = 226.74] [color={rgb, 255:red, 0; green, 0; blue, 0 }  ][line width=0.75]    (10.93,-3.29) .. controls (6.95,-1.4) and (3.31,-0.3) .. (0,0) .. controls (3.31,0.3) and (6.95,1.4) .. (10.93,3.29)   ;
\draw    (551,331.33) -- (583.57,363.27) ;
\draw [shift={(585,364.67)}, rotate = 224.44] [color={rgb, 255:red, 0; green, 0; blue, 0 }  ][line width=0.75]    (10.93,-3.29) .. controls (6.95,-1.4) and (3.31,-0.3) .. (0,0) .. controls (3.31,0.3) and (6.95,1.4) .. (10.93,3.29)   ;
\draw [color={rgb, 255:red, 74; green, 144; blue, 226 }  ,draw opacity=1 ]   (548.78,210) -- (589.59,85.9) ;
\draw [shift={(590.22,84)}, rotate = 108.2] [color={rgb, 255:red, 74; green, 144; blue, 226 }  ,draw opacity=1 ][line width=0.75]    (10.93,-3.29) .. controls (6.95,-1.4) and (3.31,-0.3) .. (0,0) .. controls (3.31,0.3) and (6.95,1.4) .. (10.93,3.29)   ;
\draw    (302,313.85) -- (334.56,282.54) ;
\draw [shift={(336,281.15)}, rotate = 136.12] [color={rgb, 255:red, 0; green, 0; blue, 0 }  ][line width=0.75]    (10.93,-3.29) .. controls (6.95,-1.4) and (3.31,-0.3) .. (0,0) .. controls (3.31,0.3) and (6.95,1.4) .. (10.93,3.29)   ;
\draw    (352,313.84) -- (385.56,281.55) ;
\draw [shift={(387,280.16)}, rotate = 136.1] [color={rgb, 255:red, 0; green, 0; blue, 0 }  ][line width=0.75]    (10.93,-3.29) .. controls (6.95,-1.4) and (3.31,-0.3) .. (0,0) .. controls (3.31,0.3) and (6.95,1.4) .. (10.93,3.29)   ;
\draw    (402,313.67) -- (434.57,281.73) ;
\draw [shift={(436,280.33)}, rotate = 135.56] [color={rgb, 255:red, 0; green, 0; blue, 0 }  ][line width=0.75]    (10.93,-3.29) .. controls (6.95,-1.4) and (3.31,-0.3) .. (0,0) .. controls (3.31,0.3) and (6.95,1.4) .. (10.93,3.29)   ;
\draw    (452,313.01) -- (485.55,281.36) ;
\draw [shift={(487,279.99)}, rotate = 136.67] [color={rgb, 255:red, 0; green, 0; blue, 0 }  ][line width=0.75]    (10.93,-3.29) .. controls (6.95,-1.4) and (3.31,-0.3) .. (0,0) .. controls (3.31,0.3) and (6.95,1.4) .. (10.93,3.29)   ;
\draw    (503,312.68) -- (533.57,282.72) ;
\draw [shift={(535,281.32)}, rotate = 135.58] [color={rgb, 255:red, 0; green, 0; blue, 0 }  ][line width=0.75]    (10.93,-3.29) .. controls (6.95,-1.4) and (3.31,-0.3) .. (0,0) .. controls (3.31,0.3) and (6.95,1.4) .. (10.93,3.29)   ;
\draw    (551,313.67) -- (583.57,281.73) ;
\draw [shift={(585,280.33)}, rotate = 135.56] [color={rgb, 255:red, 0; green, 0; blue, 0 }  ][line width=0.75]    (10.93,-3.29) .. controls (6.95,-1.4) and (3.31,-0.3) .. (0,0) .. controls (3.31,0.3) and (6.95,1.4) .. (10.93,3.29)   ;
\draw [color={rgb, 255:red, 74; green, 144; blue, 226 }  ,draw opacity=1 ]   (202,80.68) -- (232.6,111.89) ;
\draw [shift={(234,113.32)}, rotate = 225.57] [color={rgb, 255:red, 74; green, 144; blue, 226 }  ,draw opacity=1 ][line width=0.75]    (10.93,-3.29) .. controls (6.95,-1.4) and (3.31,-0.3) .. (0,0) .. controls (3.31,0.3) and (6.95,1.4) .. (10.93,3.29)   ;
\draw [color={rgb, 255:red, 126; green, 211; blue, 33 }  ,draw opacity=1 ]   (203,130.88) -- (231.61,160.68) ;
\draw [shift={(233,162.13)}, rotate = 226.17] [color={rgb, 255:red, 126; green, 211; blue, 33 }  ,draw opacity=1 ][line width=0.75]    (10.93,-3.29) .. controls (6.95,-1.4) and (3.31,-0.3) .. (0,0) .. controls (3.31,0.3) and (6.95,1.4) .. (10.93,3.29)   ;
\draw [color={rgb, 255:red, 245; green, 166; blue, 35 }  ,draw opacity=1 ]   (201,180.15) -- (233.56,211.46) ;
\draw [shift={(235,212.85)}, rotate = 223.88] [color={rgb, 255:red, 245; green, 166; blue, 35 }  ,draw opacity=1 ][line width=0.75]    (10.93,-3.29) .. controls (6.95,-1.4) and (3.31,-0.3) .. (0,0) .. controls (3.31,0.3) and (6.95,1.4) .. (10.93,3.29)   ;
\draw    (203,230.05) -- (232.63,261.49) ;
\draw [shift={(234,262.95)}, rotate = 226.7] [color={rgb, 255:red, 0; green, 0; blue, 0 }  ][line width=0.75]    (10.93,-3.29) .. controls (6.95,-1.4) and (3.31,-0.3) .. (0,0) .. controls (3.31,0.3) and (6.95,1.4) .. (10.93,3.29)   ;
\draw    (203,281.27) -- (230.64,311.26) ;
\draw [shift={(232,312.73)}, rotate = 227.34] [color={rgb, 255:red, 0; green, 0; blue, 0 }  ][line width=0.75]    (10.93,-3.29) .. controls (6.95,-1.4) and (3.31,-0.3) .. (0,0) .. controls (3.31,0.3) and (6.95,1.4) .. (10.93,3.29)   ;
\draw    (200,331.33) -- (233.57,364.27) ;
\draw [shift={(235,365.67)}, rotate = 224.45] [color={rgb, 255:red, 0; green, 0; blue, 0 }  ][line width=0.75]    (10.93,-3.29) .. controls (6.95,-1.4) and (3.31,-0.3) .. (0,0) .. controls (3.31,0.3) and (6.95,1.4) .. (10.93,3.29)   ;
\draw [color={rgb, 255:red, 189; green, 16; blue, 224 }  ,draw opacity=1 ]   (197.81,209) -- (238.56,85.9) ;
\draw [shift={(239.19,84)}, rotate = 108.32] [color={rgb, 255:red, 189; green, 16; blue, 224 }  ,draw opacity=1 ][line width=0.75]    (10.93,-3.29) .. controls (6.95,-1.4) and (3.31,-0.3) .. (0,0) .. controls (3.31,0.3) and (6.95,1.4) .. (10.93,3.29)   ;
\draw    (200,313.85) -- (232.56,282.54) ;
\draw [shift={(234,281.15)}, rotate = 136.12] [color={rgb, 255:red, 0; green, 0; blue, 0 }  ][line width=0.75]    (10.93,-3.29) .. controls (6.95,-1.4) and (3.31,-0.3) .. (0,0) .. controls (3.31,0.3) and (6.95,1.4) .. (10.93,3.29)   ;
\draw [color={rgb, 255:red, 189; green, 16; blue, 224 }  ,draw opacity=1 ]   (252,80.82) -- (285.53,111.82) ;
\draw [shift={(287,113.18)}, rotate = 222.75] [color={rgb, 255:red, 189; green, 16; blue, 224 }  ,draw opacity=1 ][line width=0.75]    (10.93,-3.29) .. controls (6.95,-1.4) and (3.31,-0.3) .. (0,0) .. controls (3.31,0.3) and (6.95,1.4) .. (10.93,3.29)   ;
\draw [color={rgb, 255:red, 74; green, 144; blue, 226 }  ,draw opacity=1 ]   (252,131.15) -- (283.56,161.47) ;
\draw [shift={(285,162.85)}, rotate = 223.85] [color={rgb, 255:red, 74; green, 144; blue, 226 }  ,draw opacity=1 ][line width=0.75]    (10.93,-3.29) .. controls (6.95,-1.4) and (3.31,-0.3) .. (0,0) .. controls (3.31,0.3) and (6.95,1.4) .. (10.93,3.29)   ;
\draw [color={rgb, 255:red, 126; green, 211; blue, 33 }  ,draw opacity=1 ]   (251,179.67) -- (285.52,210.99) ;
\draw [shift={(287,212.33)}, rotate = 222.22] [color={rgb, 255:red, 126; green, 211; blue, 33 }  ,draw opacity=1 ][line width=0.75]    (10.93,-3.29) .. controls (6.95,-1.4) and (3.31,-0.3) .. (0,0) .. controls (3.31,0.3) and (6.95,1.4) .. (10.93,3.29)   ;
\draw    (253,230.15) -- (285.56,261.46) ;
\draw [shift={(287,262.85)}, rotate = 223.88] [color={rgb, 255:red, 0; green, 0; blue, 0 }  ][line width=0.75]    (10.93,-3.29) .. controls (6.95,-1.4) and (3.31,-0.3) .. (0,0) .. controls (3.31,0.3) and (6.95,1.4) .. (10.93,3.29)   ;
\draw    (252,281.5) -- (282.59,312.09) ;
\draw [shift={(284,313.5)}, rotate = 225] [color={rgb, 255:red, 0; green, 0; blue, 0 }  ][line width=0.75]    (10.93,-3.29) .. controls (6.95,-1.4) and (3.31,-0.3) .. (0,0) .. controls (3.31,0.3) and (6.95,1.4) .. (10.93,3.29)   ;
\draw    (250,330.85) -- (285.53,363.79) ;
\draw [shift={(287,365.15)}, rotate = 222.84] [color={rgb, 255:red, 0; green, 0; blue, 0 }  ][line width=0.75]    (10.93,-3.29) .. controls (6.95,-1.4) and (3.31,-0.3) .. (0,0) .. controls (3.31,0.3) and (6.95,1.4) .. (10.93,3.29)   ;
\draw [color={rgb, 255:red, 245; green, 166; blue, 35 }  ,draw opacity=1 ]   (247.91,210) -- (290.45,84.89) ;
\draw [shift={(291.09,83)}, rotate = 108.78] [color={rgb, 255:red, 245; green, 166; blue, 35 }  ,draw opacity=1 ][line width=0.75]    (10.93,-3.29) .. controls (6.95,-1.4) and (3.31,-0.3) .. (0,0) .. controls (3.31,0.3) and (6.95,1.4) .. (10.93,3.29)   ;
\draw    (250,314.15) -- (285.53,281.21) ;
\draw [shift={(287,279.85)}, rotate = 137.16] [color={rgb, 255:red, 0; green, 0; blue, 0 }  ][line width=0.75]    (10.93,-3.29) .. controls (6.95,-1.4) and (3.31,-0.3) .. (0,0) .. controls (3.31,0.3) and (6.95,1.4) .. (10.93,3.29)   ;

\end{tikzpicture}

    \end{center}
~\\ For the sink $c$, notice that $d_c:\{c_s  \}\rightarrow \mathbb{N}$ is defined as $d(c_s)=|s|+1$ for all $s<1$ and $d(c_s)=0$ of all $s\geq 1$, an eventually trivial map. Now, for the vertices $v_i$ in the maximal cycle, 
$d((v_i)_1)=i:=k_i$ ($i=1,2,3,4$). Hence, the integer $t=1$ in Theorem \ref{theoremy1} (see proof). Now, there is only one maximal cycle which is of length $4$. We verify the formula with a few examples. 

Consider $j=1$ and $s=3$. Then $i\equiv 1+1-3\equiv -1\equiv 3 \pmod{4}$. Thus, $d((v_1)_3)=k_3=3$. Consider $j=4$ and $s=5$. Then $i\equiv 4+1-5\equiv 0\equiv 4\pmod{4}$. Thus, $d((v_4)_5)=k_4=4$. 
\end{example}

The following corollary is a direct consequence of Theorem \ref{theoremy1} and  Corollary \ref{corollaryy2}.


\begin{cor}
    \label{corollaryx1.2} 
Let $E$ be a finite graph with no sources and sinks and let $\mathcal{D}$ be the set of maximal cycles in $E$. 
Then there is a one-to-one correspondence between the set of  vector space distributions of  representations in $\ffRep(\overline{\hat{E}}, \overline{r_{ck}})$ and $\displaystyle\bigoplus_{C\in \mathcal{D}} \mathbb{N}^{|C|}.$
\end{cor}

Define $\sim$ on the set of cycles in $E$ by $C\sim D$ if $D$ is a \emph{shift} of $C$, that is if $C=e_1e_2\dots  e_n$ then $D=e_ie_{i+1}\dots e_ne_1 \dots e_{i-1}$, for some $1\leq i\leq n$. Then $\sim$ is an equivalence relation and denote by $[C]$ the equivalence class of $C$ and let $\mathcal{M}(E):= \{[C]\mid C\text{ is a maximal cycle in } E\}$. 

\begin{thm}\label{theoremfinalsss2}
 Let $E$ be a finite graph with no sources. Then there is a one-to-one correspondence between the set of vector space distributions in $\fRep^\mathbb{Z}(\overline{\hat{E}}, \overline{r_{ck}})$ and the collection 
\begin{multline*}
\mathcal{N}= \Big \{ \left((\phi_v)_{v\in I}, (\tau_w)_{w\in S}, (k_1,k_2,\cdots,k_{|C|})_{C\in \mathcal{M}(E)} \right ) ~\mid~ \\ \phi_v, \tau_w: \mathbb{Z}\rightarrow \mathbb{N}, ~\tau_w  \textnormal{ is eventually trivial},~ k_i \in \mathbb{N} \Big \}, 
\end{multline*}
 where $I$ and $S$ are the sets of isolated vertices and sinks, respectively. 
\end{thm}

\begin{proof}

For any vector space distribution, by Theorems~\ref{theorem7b} and \ref{theoremy1}, we indeed obtain a unique element of $\mathcal{N}$. 

Suppose we are given the collection $\mathcal{N}$. We will construct a vector space distribution for a  $\mathbb{Z}$-locally finite representation $\varphi$ obtained from each of the elements in $\mathcal{N}$.

For each isolated vertex $u$, we have a mapping $\psi_u: \mathbb{Z} \rightarrow \mathbb{N}$. Define
 $$\varphi_u:\{u_i\}_{i\in \mathbb{Z}}
 \rightarrow \fvc \K \text{~~by~~} \varphi_u(u_i)=\K^{\psi_u(i)}.$$ For each sink $v$, we have an eventually trivial map $\tau_v:\mathbb{Z} \rightarrow \mathbb{N}$. Define 
 $$\varphi_v:\{v_i\}_{i\in \mathbb{Z}}\rightarrow \fvc \K \text{~~by~~} \varphi_v(v_i)=\K^{\tau_v(i)}.$$ 
 Now, for each $v\in \textnormal{Sink}(E)$, there exists $t_v\in \mathbb{Z}$ such that $\tau_v(i)=0$ for all $i\geq t_v$. We let $t=\max \{ t_v \mid v\in \textnormal{Sink}(E) \}$.
For each $w\not \in C(E)\cup \textnormal{Sink}(E)$, we define 
 $$\varphi_w:\{w_i\}_{i\in \mathbb{Z}}\rightarrow \fvc \K \text{~~by~~} \varphi_w(w_i)=0 \text{~for~all~} i\geq t.$$ 
Accordingly, we have $\varphi(w_i)=0$, for all $i\geq t$ and for all non-isolated vertex $w\not \in C(E)$.

Let $C=e_0e_1\cdots e_{n-1}$ be a maximal cycle and $C^0=\{v_0, v_1, \cdots , v_{n-1}  \}$ with $s(e_{i})=v_{i}=r(e_{i-1})$, where $0< i < n-1$ and $s(e_{n-1})=v_{n-1}$ and $r(e_{n-1})=s(e_0)=v_0$.
 For each $i=0,1,\cdots,n-1$, let $p_i$ be the path in $\overline{\hat{E}}$ obtained from the edges $e_0, e_1,\cdots, e_{n-1}$ for which $s(p_i)=(v_i)_t$. That is, 
$$p_i=(e_i)_t(e_{i+1})_{t+1}(e_{i+2})_{t+2}\cdots (e_{n-1})_{t+n-1-i}(e_0)_{t+n-i}(e_1)_{t+n-i+1}\cdots . $$ 
Since $C$ is a cycle, the paths $p_i$ are mutually disjoint, that is, ${p_i}^0\cap {p_j}^0=\varnothing$, for all $i\neq j$. Now, for the equivalence class of the maximal cycle $C$, there exists $(k_0, k_1, \cdots, k_{n-1})\in \mathbb{N}^n$. We define $$\varphi_C:=\varphi_{C^0}\mid \{(v_i)_j \mid i< n, j\in \mathbb{Z}\}\rightarrow \fvc \K ~~\text{~~by~~}~~ \varphi_C(u)=\K^{k_i},$$
for every $u\in p_i^0$ and  $i <n$. 
That is in particular, $\varphi_C((v_i)_t)=\varphi_C(s(p_i))=\K^{k_i}$. 
Using similar arguments as in the proof of Theorem \ref{theoremy1}, for every $s\geq t$ and $i,j=0,1,\cdots, n-1$, we have
    $$\varphi_C((v_i)_s)=\K^{k_j}, \text{~~for~~}j\equiv i+t-s\pmod n.$$
We have defined $\varphi$ for $s\geq t$ for all vertices and for all $i\in \mathbb{Z}$ for sinks and isolated vertices. We now show this is enough and $\varphi:\overline{\hat{E}}^0\rightarrow \mathbb{N}$ defines a vector space distribution satisfying $\overline{r_{ck}}$ in $\overline{\hat{E}}$.

Let $w\in E^0$ and $s \geq t$. Then $d(r(e)_{s+1})=0$, for all $e\in s^{-1}({w})\setminus r^{-1}(C(E))$. If $w\not \in C(E)$, then $r(e)\not \in C(E)$, for every $e\in s^{-1}(w)$. Thus, $d(w_s)=0$ and $d(r(e)_{s+1})=0$, for all $e\in s^{-1}(w)$. Hence, $$d(w_s)=0=\sum_{e\in s^{-1}({w})} d(r(e)_{s+1}).$$ 
If $w\in C(E)$, then there exists a maximal cycle $D=d_0d_1\cdots d_{n-1}$ with $w\in D^0$. Let $D^0=\{u_0,u_1, \cdots, u_{n-1} \}$ such that $s(d_i)=u_i=r(d_{i-1}) 
 \pmod{n}$. Then $w=u_q$, for some $0\leq q\leq n-1$. Without loss of generality, we assume that $q<n-1$. For the equivalence class of the maximal cycle $D$, there exists $(m_0,m_1,\cdots , m_{n-1})\in \mathbb{N}^n$ and have 
$$\varphi_D((u_i)_s)=\K^{m_j} \text{~~for~~}j\equiv i+t-s\pmod{n},$$ for all $s\geq t$ and $0\leq i,j\leq n-1$.
By definition, $D$ is not connected to any other maximal cycle. In particular, we have $$r(s^{-1}(\{w\}))\cap C(E)=r(s^{-1}(\{u_q\}))\cap C(E)=\{u_{q+1}\}.$$
This implies that for every $e\in s^{-1}(w)$ with $r(e)\neq u_{q+1}$,  we have $d(r(e)_{s+1})=0$. 

Now,  
$\varphi_D((u_q)_s)=\K^{m_j}$, for some $j\equiv q+t-s \pmod n$. Then $j\equiv (q+1)+t-(s+1) \pmod n$. Hence, $\varphi_D((u_{q+1})_{s+1})=\K^{m_j}$, that is, $d((u_{q+1})_{s+1})=m_j=d((u_q)_s)$. Hence, 
\begin{align*}
d(w_s) &= d((u_q)_s)\\&= d((u_{q+1})_{s+1})\\&= d((u_{q+1})_{s+1})+\sum \{ d(r(e)_{s+1}):{e\in s^{-1}({w}), r(e)\neq u_{q+1} }\}\\
&= \sum_{e\in s^{-1}({w})} d(r(e)_{s+1}).
\end{align*}
Hence, for all $s\geq t$, the necessary conditions for a vector space distribution of a representation is satisfied. 

Let $w\in E^0$ and we consider $w_{t-1}$. If $w$ is a sink, then $d(w_{t-1})=\tau_w(t-1)$. If $w$ is not a sink, then $s^{-1}(w)=\{f_1, f_2, \cdots, f_s\}$, for some $f_i\in  E^1$. Define
$$\varphi(w_{t-1})=\bigoplus_{i=1}^s \varphi(r(f_i)_t). $$
Note that $\varphi(r(f_i)_t)$ is determined for all $i=1,2,\cdots , s$ and 
$$d(w_{t-1})=\sum_{i=1}^s d(r(f_i)_t). $$
Hence, $w_{t-1}$ is determined for all $w\in E^0$. Similar process for $w_{t-2}$ and continuing, we shall have  
$$d(w_{s})=\sum_{e\in s^{-1}(v)} d(r(e)_{s+1}), $$
for all $s\in \mathbb{Z}$ and $w\in E^0$. Accordingly, we obtain a suitable vector space distribution. \end{proof}

\begin{example}
    Consider the graph $E$ below. Define $\psi_x:\mathbb{Z}\rightarrow \mathbb{N}$ by $\psi_x(i)=|i|$,  for all $i\in \mathbb Z$, $\tau_w: \mathbb Z \rightarrow \mathbb N$ be an eventually trivial mapping with $\tau_w(i)=1$, for all $i\leq 0$ and $\tau_w(i)=0$, for all $i>0$, and for the maximal cycle with vertices $\{v,u\}$, let us consider $(2,3)\in \mathbb N^2$.   We show that given these mappings and tuple, we can obtain a vector space distribution of a representation in $\fRep^\mathbb{Z}(\overline{\hat{E}}, \overline{r_{ck}})$.

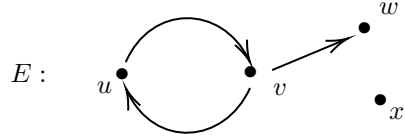
\tikzset{every picture/.style={line width=0.75pt}} 
\begin{center}

\begin{tikzpicture}[x=0.75pt,y=0.75pt,yscale=-1,xscale=1]

\draw    (65,45) .. controls (77.74,15.6) and (112.57,15.97) .. (126.19,43.29) ;
\draw [shift={(127,45)}, rotate = 245.85] [color={rgb, 255:red, 0; green, 0; blue, 0 }  ][line width=0.75]    (10.93,-3.29) .. controls (6.95,-1.4) and (3.31,-0.3) .. (0,0) .. controls (3.31,0.3) and (6.95,1.4) .. (10.93,3.29)   ;
\draw    (127,58) .. controls (114.26,86.42) and (79.43,88.91) .. (65.81,59.82) ;
\draw [shift={(65,58)}, rotate = 67.25] [color={rgb, 255:red, 0; green, 0; blue, 0 }  ][line width=0.75]    (10.93,-3.29) .. controls (6.95,-1.4) and (3.31,-0.3) .. (0,0) .. controls (3.31,0.3) and (6.95,1.4) .. (10.93,3.29)   ;
\draw    (138,49) -- (174.16,33.78) ;
\draw [shift={(176,33)}, rotate = 157.17] [color={rgb, 255:red, 0; green, 0; blue, 0 }  ][line width=0.75]    (10.93,-3.29) .. controls (6.95,-1.4) and (3.31,-0.3) .. (0,0) .. controls (3.31,0.3) and (6.95,1.4) .. (10.93,3.29)   ;

\draw (49,53.4) node [anchor=north west][inner sep=0.75pt]  [font=\normalsize]  {$u$};
\draw (137,54.4) node [anchor=north west][inner sep=0.75pt]    {$v$};
\draw (190,13.4) node [anchor=north west][inner sep=0.75pt]    {$w$};
\draw (195,66.4) node [anchor=north west][inner sep=0.75pt]    {$x$};
\draw (122,45.4) node [anchor=north west][inner sep=0.75pt]  [font=\large]  {$\bullet $};
\draw (6,44.16) node [anchor=north west][inner sep=0.75pt]    {$E:$};
\draw (58,46.4) node [anchor=north west][inner sep=0.75pt]  [font=\large]  {$\bullet $};
\draw (179,23.4) node [anchor=north west][inner sep=0.75pt]  [font=\large]  {$\bullet $};
\draw (187,59.4) node [anchor=north west][inner sep=0.75pt]  [font=\large]  {$\bullet $};

\end{tikzpicture}
\end{center}

Now, we define $\varphi:E\rightarrow \vc \K$ by having $\varphi(v_0)=\K^2$ and $\varphi(u_0)=\K^3$, and for all $i$, $\varphi(w_i)=\K^{\tau_w{(i)}}$ and $\varphi(x_i)=\K^{\psi_x(i)}$. As seen in the figure below, we obtain a possible vector space distribution.

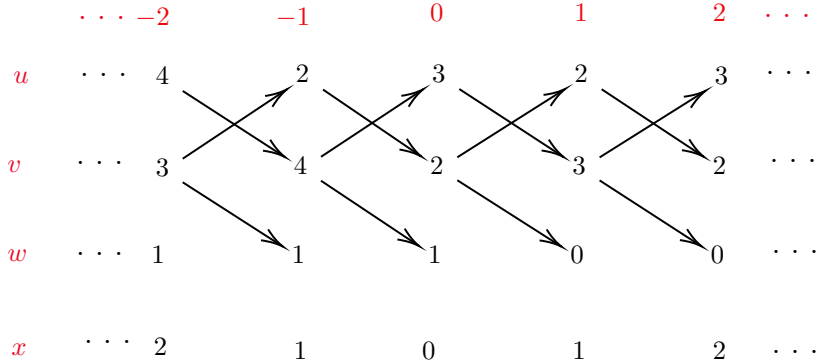
\begin{center}

\tikzset{every picture/.style={line width=0.75pt}} 

\begin{tikzpicture}[x=0.75pt,y=0.75pt,yscale=-1,xscale=1]

\draw (103.6,57.69) node  [font=\Large]  {$.\ .\ .$};
\draw (101.6,101.69) node  [font=\Large]  {$.\ .\ .$};
\draw (101.6,147.69) node  [font=\Large]  {$.\ .\ .$};
\draw (105.6,191.69) node  [font=\Large]  {$.\ .\ .$};
\draw (445.6,56.75) node  [font=\Large]  {$.\ .\ .$};
\draw (447.6,100.75) node  [font=\Large]  {$.\ .\ .$};
\draw (448.6,146.75) node  [font=\Large]  {$.\ .\ .$};
\draw (448.6,196.75) node  [font=\Large]  {$.\ .\ .$};
\draw (198,51.4) node [anchor=north west][inner sep=0.75pt]    {$2$};
\draw (128,52.4) node [anchor=north west][inner sep=0.75pt]    {$4$};
\draw (266,51.4) node [anchor=north west][inner sep=0.75pt]    {$3$};
\draw (337,51.4) node [anchor=north west][inner sep=0.75pt]    {$2$};
\draw (407,52.4) node [anchor=north west][inner sep=0.75pt]    {$3$};
\draw (197,97.4) node [anchor=north west][inner sep=0.75pt]    {$4$};
\draw (128,98.4) node [anchor=north west][inner sep=0.75pt]    {$3$};
\draw (265,97.4) node [anchor=north west][inner sep=0.75pt]    {$2$};
\draw (336,97.4) node [anchor=north west][inner sep=0.75pt]    {$3$};
\draw (406,97.4) node [anchor=north west][inner sep=0.75pt]    {$2$};
\draw (196,142.4) node [anchor=north west][inner sep=0.75pt]    {$1$};
\draw (126,142.4) node [anchor=north west][inner sep=0.75pt]    {$1$};
\draw (264,142.4) node [anchor=north west][inner sep=0.75pt]    {$1$};
\draw (335,142.4) node [anchor=north west][inner sep=0.75pt]    {$0$};
\draw (405,142.4) node [anchor=north west][inner sep=0.75pt]    {$0$};
\draw (197,190.4) node [anchor=north west][inner sep=0.75pt]    {$1$};
\draw (127,188.4) node [anchor=north west][inner sep=0.75pt]    {$2$};
\draw (261,190.4) node [anchor=north west][inner sep=0.75pt]    {$0$};
\draw (336,190.4) node [anchor=north west][inner sep=0.75pt]    {$1$};
\draw (406,190.4) node [anchor=north west][inner sep=0.75pt]    {$2$};
\draw (188,22.4) node [anchor=north west][inner sep=0.75pt]  [color={rgb, 255:red, 237; green, 9; blue, 13 }  ,opacity=1 ]  {$-1$};
\draw (265,20.4) node [anchor=north west][inner sep=0.75pt]  [color={rgb, 255:red, 237; green, 9; blue, 13 }  ,opacity=1 ]  {$0$};
\draw (337,20.4) node [anchor=north west][inner sep=0.75pt]  [color={rgb, 255:red, 237; green, 9; blue, 13 }  ,opacity=1 ]  {$1$};
\draw (406,20.4) node [anchor=north west][inner sep=0.75pt]  [color={rgb, 255:red, 237; green, 9; blue, 13 }  ,opacity=1 ]  {$2$};
\draw (118,22.4) node [anchor=north west][inner sep=0.75pt]  [color={rgb, 255:red, 237; green, 9; blue, 13 }  ,opacity=1 ]  {$-2$};
\draw (57,54.4) node [anchor=north west][inner sep=0.75pt]  [color={rgb, 255:red, 237; green, 7; blue, 34 }  ,opacity=1 ]  {$u$};
\draw (54,99.4) node [anchor=north west][inner sep=0.75pt]  [color={rgb, 255:red, 237; green, 7; blue, 34 }  ,opacity=1 ]  {$v$};
\draw (54,144.4) node [anchor=north west][inner sep=0.75pt]  [color={rgb, 255:red, 237; green, 7; blue, 34 }  ,opacity=1 ]  {$w$};
\draw (56,191.4) node [anchor=north west][inner sep=0.75pt]  [color={rgb, 255:red, 237; green, 7; blue, 34 }  ,opacity=1 ]  {$x$};
\draw (102.6,28.69) node  [font=\Large,color={rgb, 255:red, 237; green, 7; blue, 34 }  ,opacity=1 ]  {$.\ .\ .$};
\draw (444.6,27.75) node  [font=\Large,color={rgb, 255:red, 237; green, 7; blue, 34 }  ,opacity=1 ]  {$.\ .\ .$};
\draw    (143,65.87) -- (192.32,98.04) ;
\draw [shift={(194,99.13)}, rotate = 213.11] [color={rgb, 255:red, 0; green, 0; blue, 0 }  ][line width=0.75]    (10.93,-3.29) .. controls (6.95,-1.4) and (3.31,-0.3) .. (0,0) .. controls (3.31,0.3) and (6.95,1.4) .. (10.93,3.29)   ;
\draw    (213,65.18) -- (260.35,97.69) ;
\draw [shift={(262,98.82)}, rotate = 214.47] [color={rgb, 255:red, 0; green, 0; blue, 0 }  ][line width=0.75]    (10.93,-3.29) .. controls (6.95,-1.4) and (3.31,-0.3) .. (0,0) .. controls (3.31,0.3) and (6.95,1.4) .. (10.93,3.29)   ;
\draw    (281,64.91) -- (331.33,97.99) ;
\draw [shift={(333,99.09)}, rotate = 213.31] [color={rgb, 255:red, 0; green, 0; blue, 0 }  ][line width=0.75]    (10.93,-3.29) .. controls (6.95,-1.4) and (3.31,-0.3) .. (0,0) .. controls (3.31,0.3) and (6.95,1.4) .. (10.93,3.29)   ;
\draw    (352,65) -- (401.34,97.89) ;
\draw [shift={(403,99)}, rotate = 213.69] [color={rgb, 255:red, 0; green, 0; blue, 0 }  ][line width=0.75]    (10.93,-3.29) .. controls (6.95,-1.4) and (3.31,-0.3) .. (0,0) .. controls (3.31,0.3) and (6.95,1.4) .. (10.93,3.29)   ;
\draw    (143,99.96) -- (193.34,66.16) ;
\draw [shift={(195,65.04)}, rotate = 146.12] [color={rgb, 255:red, 0; green, 0; blue, 0 }  ][line width=0.75]    (10.93,-3.29) .. controls (6.95,-1.4) and (3.31,-0.3) .. (0,0) .. controls (3.31,0.3) and (6.95,1.4) .. (10.93,3.29)   ;
\draw    (212,99) -- (261.34,66.11) ;
\draw [shift={(263,65)}, rotate = 146.31] [color={rgb, 255:red, 0; green, 0; blue, 0 }  ][line width=0.75]    (10.93,-3.29) .. controls (6.95,-1.4) and (3.31,-0.3) .. (0,0) .. controls (3.31,0.3) and (6.95,1.4) .. (10.93,3.29)   ;
\draw    (280,99.25) -- (332.31,65.83) ;
\draw [shift={(334,64.75)}, rotate = 147.43] [color={rgb, 255:red, 0; green, 0; blue, 0 }  ][line width=0.75]    (10.93,-3.29) .. controls (6.95,-1.4) and (3.31,-0.3) .. (0,0) .. controls (3.31,0.3) and (6.95,1.4) .. (10.93,3.29)   ;
\draw    (143,111.82) -- (191.32,143.09) ;
\draw [shift={(193,144.18)}, rotate = 212.91] [color={rgb, 255:red, 0; green, 0; blue, 0 }  ][line width=0.75]    (10.93,-3.29) .. controls (6.95,-1.4) and (3.31,-0.3) .. (0,0) .. controls (3.31,0.3) and (6.95,1.4) .. (10.93,3.29)   ;
\draw    (212,111.04) -- (259.34,142.84) ;
\draw [shift={(261,143.96)}, rotate = 213.89] [color={rgb, 255:red, 0; green, 0; blue, 0 }  ][line width=0.75]    (10.93,-3.29) .. controls (6.95,-1.4) and (3.31,-0.3) .. (0,0) .. controls (3.31,0.3) and (6.95,1.4) .. (10.93,3.29)   ;
\draw    (280,110.79) -- (330.32,143.13) ;
\draw [shift={(332,144.21)}, rotate = 212.74] [color={rgb, 255:red, 0; green, 0; blue, 0 }  ][line width=0.75]    (10.93,-3.29) .. controls (6.95,-1.4) and (3.31,-0.3) .. (0,0) .. controls (3.31,0.3) and (6.95,1.4) .. (10.93,3.29)   ;
\draw    (351,110.87) -- (400.32,143.04) ;
\draw [shift={(402,144.13)}, rotate = 213.11] [color={rgb, 255:red, 0; green, 0; blue, 0 }  ][line width=0.75]    (10.93,-3.29) .. controls (6.95,-1.4) and (3.31,-0.3) .. (0,0) .. controls (3.31,0.3) and (6.95,1.4) .. (10.93,3.29)   ;
\draw    (351,99.3) -- (402.31,66.77) ;
\draw [shift={(404,65.7)}, rotate = 147.63] [color={rgb, 255:red, 0; green, 0; blue, 0 }  ][line width=0.75]    (10.93,-3.29) .. controls (6.95,-1.4) and (3.31,-0.3) .. (0,0) .. controls (3.31,0.3) and (6.95,1.4) .. (10.93,3.29)   ;

\end{tikzpicture}

\end{center}
\end{example}

Recall that the Graded Classification Conjecture~\cite{mathann, hazdyn} states that any order-preserving $\mathbb Z[x,x^{-1}]$-module isomorphism 
$K_0^{\gr}(L_\K(E)) \cong K_0^{\gr}(L_\K(F))$ gives a graded category equivalence $\Gr L_\K(E) \approx_{\gr} \Gr L_\K(F)$.  We are in a position to prove our main theorem. 

\begin{thm}\label{mainthfer}
Let $E$ and $F$ be finite graphs with no sources and sinks and let $L_\K(E)$ and $L_\K(F)$ be their associated Leavitt path algebras with coefficient in a field $\K$, respectively.  If there is an order-preserving $\mathbb Z[x,x^{-1}]$-module isomorphism $K_0^{\gr}(L_\K(E)) \cong K_0^{\gr}(L_\K(F))$,  then there is an equivalence, 
\[\fModd L_\K(E) \approx \fModd L_\K(F)\] as well as graded equivalences
\[\fGr L_\K(E) \approx_{\gr} \fGr L_\K(F) \text { and }  \fGr[\mathbb{Z}]L_\K(E)\approx_{\gr}
    \fGr[\mathbb{Z}] L_\K(F).\]

\end{thm}
\begin{proof}
Since $K_0^{\gr}(L_\K(E)) \cong K_0^{\gr}(L_\K(F))$ as an order-preserving $\mathbb Z[x,x^{-1}]$-module isomorphism then \[T_E\cong \mathcal V^{\gr}(L_\K(E))  \cong  \mathcal V^{\gr}(L_\K(F)) \cong T_F,\] as $\mathbb Z$-monoids (see~\S\ref{lpat}). 

  Since $T_E\cong T_F$,  by Theorem~\ref{talmax}, there is a one-to-one correspondence between the maximal cycles $E$ and $F$ with the same length. Let $\{C_1,\dots,C_k\}$ and  $\{C_1',\dots,C_k'\}$ be the sets of maximal cycles in $E$ and $F$, respectively, where $C_i$ corresponds to $C_i'$, $1\leq i \leq k$. Thus $|C_i| = |C_i'|$. 

Next we show that there is an equivalence of categories $\fRep(\hat E, \rr) \cong \fRep(\hat F, \rr)$. Consider  $\rho: \hat E \rightarrow \vc \K$ in  $\fRep(\hat E, \rr)$. By Proposition~\ref{firstmanin}, for $v\in \hat E^0$  we have 
\begin{equation*}
\rho(v)=
 \left\{ \begin{array}{ll}
 \K^{n_{C_i}} & \textrm{if $v\in C_i$},\\
  0 & \textrm{otherwise,}
  \end{array} \right.
\end{equation*}
for a fixed $n_{C_i}\in \mathbb N$. 

Define  $\rho': \hat F \rightarrow \vc \K$ in  $\fRep(\hat F, \rr)$ as follows: for $v\in \hat F^0$, set 
 \begin{equation*}
\rho'(v)=
 \left\{ \begin{array}{ll}
 \K^{n_{C_i}} & \textrm{if $v\in C'_i$},\\
 0 & \textrm{otherwise.}
  \end{array} \right.
\end{equation*}
i.e., the distribution of vector spaces on the vertices of $F$ are exactly those of $E$. For the linear transformations, we assign the same linear transformations in $\rho$ to  the corresponding source and range vector space in $\rho'$. The correspondence 
\begin{align*}
\fRep(\hat E, \rr) &\longrightarrow \fRep(\hat F, \rr),\\
\rho &\longmapsto \rho'
\end{align*}
induces a category equivalence (indeed a category isomorphism). By Propositions~\ref{mainrain} for the case of Leavitt path algebras (see~\S\ref{lbvyfgjdh}), we have 
\[\fModd L_\K(E) \cong \fRep(\hat E, \rr) \cong \fRep(\hat F, \rr)\cong \fModd L_\K(F).\]

Similarly one can establish a graded isomorphism $\fRep(\overline{\hat{E}}, \overline \rr) \cong \fRep(\overline{\hat{F}}, \overline \rr)$ and there  a graded  equivalence 
\[\fGr L_\K(E) \approx \fRep(\overline{\hat{E}}, \overline \rr) \cong \fRep(\overline{\hat{F}}, \overline \rr) \cong \fGr L_\K(F).\] 

For the final equivalence, for each $i$, let $C_i^0=\{z_1^{(i)}, z_2^{(i)}, \cdots, z_{n_i}^{(i)}\}$ and ${C_{i}'}^0=\{w_1^{(i)}, w_2^{(i)}, \cdots, w_{n_i}^{(i)}\}$. 
Let $\varphi \in \fRep^{\mathbb{Z}}(\overline{\hat{E}}, \overline{r_{ck}})$. Then by Corollary \ref{corollaryy2}, $\varphi_v=0$, for all $v\not \in C(E)$. Now for each $i$, let $(k_1^{(i)}, k_2^{(i)}, \cdots , k_{n_i}^{(i)})$ be the $n_i$-tuple as in Theorem \ref{theoremfinalsss2}. That is, in Theorem \ref{theoremy1}, $d_{\varphi}((z_j^{(i)})_t)=k_j^{(i)}$, for all $t\in \mathbb{Z}$. Define $\varphi':\overline{\hat{F}}\rightarrow \fvc \K$ with $d_{\varphi'}((w_j^{(i)})_t)=d_{\varphi}((z_j^{(i)})_t)=k_j^{(i)}$ and such that each linear map in $\varphi$ be carried in $\varphi'$ on corresponding source and range vector space. Also, by Corollary \ref{corollaryy2}, $\varphi'_v=0$, for all $v\not \in C(F)$. Thus, this gives us an equivalence of categories $\fRep^{\mathbb{Z}}(\overline{\hat{E}}, \overline{r_{ck}})\approx \fRep^{\mathbb{Z}}(\overline{\hat{F}}, \overline{r_{ck}})$. Combining this with Proposition~\ref{mainlemma1} for the case of Leavitt path algebras, gives a graded equivalence of categories $$\fGr[\mathbb{Z}] L_\K(E)\approx_{\gr} \fGr[\mathbb{Z}]L_\K(F).$$ 
\end{proof}

\begin{rmk}
In fact for the equivalences $\fModd L_\K(E) \approx \fModd L_\K(F)$ as well as the graded equivalences
$\fGr L_\K(E) \approx_{\gr} \fGr L_\K(F)$, Theorem~\ref{mainthfer} is valid for graphs where sources are isolated vertices, i.e, sinks are allowed. The statements needed to prove these facts are all written with this assumption. However for the equivalence   $\fGr[\mathbb{Z}]L_\K(E)\approx_{\gr} \fGr[\mathbb{Z}] L_\K(F),$ we can only establish the Theorem~\ref{mainthfer} for the case of essential graphs (i.e., no sinks and sources). 
    \end{rmk}

We thus get the following corollary, providing further evidence that the graded $K$-group (and the talented monoid) captures substantial information about the algebra structure.

\begin{cor}\label{hfgfhfdd}
Let $E$ and $F$ be finite graphs with no sources and sinks. If there is an order-preserving $\mathbb Z[x,x^{-1}]$-module isomorphism $K_0^{\gr}(L_\K(E)) \cong K_0^{\gr}(L_\K(F))$,  then 
\begin{enumerate}[\upshape(1)]
\item There is a one-to-one correspondence between finite dimensional graded irreducible representations of $L_\K(E)$ and $L_\K(F)$.

\item There is a one-to-one correspondence between finite dimensional non-graded irreducible representations of $L_\K(E)$ and $L_\K(F)$.

\item There is a one-to-one correspondence between finite dimensional (graded) representations of $L_\K(E)$ and $L_\K(F)$.

\item There is a one-to-one correspondence between locally finite graded representations of $L_\K(E)$ and $L_\K(F)$.

\end{enumerate}

\end{cor}

\section*{Acknowledgement}

Hazrat acknowledges Australian Research Council grant DP230103184.


\begin{thebibliography}{9999}

  
\bibitem{lpabook} G. Abrams, P. Ara, M. Siles Molina, Leavitt path algebras, Lecture Notes in Mathematics {\bf 2191}, Springer, 2017.

\bibitem{abconj} G. Abrams, E. Ruiz, M. Tomforde, \emph{Recasting the Hazrat Conjecture: Relating Shift Equivalence to Graded Morita Equivalence}, 	
Algebr. Represent. Theory, \textbf{27} (2024), 1477--1511.



\bibitem{arapardo} P. Ara, E. Pardo, \emph{Towards a K-theoretic characterization of graded isomorphisms between Leavitt path algebras}, J. K-Theory {\bf 14} (2014), 
203--245. 


\bibitem{ara2006} P. Ara, M.A. Moreno, E. Pardo, \emph{Nonstable K-theory for graph algebras}, Algebras and Representation Theory
$\mathbf{10}$ (2007), 157--178.


\bibitem{arali} P. Ara, R. Hazrat, H. Li, A. Sims, \emph{Graded Steinberg algebras and their representations}, Algebra \& Number Theory {\bf 12} (2018), no. 1, 131--172.



\bibitem{guido} G.  Arnone, \emph{Lifting morphisms between graded Grothendieck groups of Leavitt path algebras}, J. of Algebra, \textbf{631} (2023), 804--829.



\bibitem{guidohom} G.  Arnone, \emph{Graded homotopy classification of Leavitt path algebras}, 	arXiv:2309.06312 [math.KT].

\bibitem{guidowillie} G. Arnone, G. Corti\~nas, \emph{Graded K-theory and Leavitt path algebras}, 
J. Algebraic Combin. {\bf 58} (2023), no. 2, 399--434.

\bibitem{bilich} B. Bilich, A. Dor-On, E. Ruiz, \emph{Equivariant homotopy classification of graph C* algebras}, 	arXiv:2408.09740 [math.OA]


\bibitem{bergman74} G.M. Bergman, \emph{Coproducts and some universal ring constructions}, Trans. Amer. Math. Soc. {\bf 200} (1974) 33--88.

\bibitem{alfi2} W. Bock, A. Sebandal,  \emph{An adjacency matrix perspective of talented monoids and Leavitt path
algebras,} Linear Algebra and Its Applications, \textbf{678} (2023), 295-316.



\bibitem{alfi} W. Bock, A. Sebandal, J. Vilela,  \emph{A talented monoid view on Lie bracket algebras over Leavitt path algebras,} Journal of Algebra and Its Applications, (2022) no. 8.
















\bibitem{carlsen} T.M. Carlsen, A. Dor-On, S. Eilers, \emph{Shift equivalences through the lens of Cuntz-Krieger algebras}, 
Analysis and PDE \textbf{17} (2024) 345--377.

\bibitem{Luiz}  L. Cordeiro, D. Gon\c{c}alves, R. Hazrat, \emph{The talented monoid of a directed graph with applications to graph algebras}, 
Revista Matem\'atica Iberoamericana, {\bf 38} (2022), no. 1, 223--256.


\bibitem{willie}  G. Corti\~nas, R. Hazrat, \emph{Classification conjectures for Leavitt path algebras},  to appear in Bulletin of LMS,  arXiv:2401.04262v2.
 

\bibitem{eilers} S. Eilers, E. Ruiz, A. Sims, \emph{Amplified graph $C^*$-algebras II: reconstruction}, (2021) arXiv:2007.00853 [math.OA]

\bibitem{green} E.L. Green, \emph{Graphs with relations, coverings and group-graded algebras}, Trans. Amer. Math. Soc. {\bf 279} (1983), no. 1, 297--310.

\bibitem{greenmarco} E.L. Green, E. Marcos, \emph{Graded quotients of path algebras: a local theory}, J. Pure Appl. Algebra {\bf 93} (1994), no. 2, 
195--226.


\bibitem{mathann} R. Hazrat, \emph{The graded Grothendieck group and classification of Leavitt path algebras},  Math. Ann. {\bf 355} (2013)  273--325. 

\bibitem{hazmark} R. Hazrat, \emph{The graded Grothendieck group as a classification tool for algebras, current status}, (2014) https://marktomforde.com/graph-algebra-problems/Hazrat-Questions.pdf 


\bibitem{hazdyn} R. Hazrat, \emph{The dynamics of Leavitt path algebras}, J. Algebra {\bf 384} (2013), 242--266. 

\bibitem{hazi} R. Hazrat,  Graded rings and graded Grothendieck groups, volume 435 of London Mathematical Society Lecture
Note Series. Cambridge University Press, Cambridge, 2016.


\bibitem{hazli} R. Hazrat, H. Li, \emph{The talented monoid of a Leavitt path algebra}, Journal of Algebra {\bf 547} (2020) 430--455.

\bibitem{lia} R. Hazrat, L. Vas,  \emph{Comparability in the graph monoid}, New York J. of Math, {\bf 26} (2020), 1375--1421. 


 
  
 \bibitem{koc1} A. Ko\c{c},  M.  \"{O}zayd\i n, \emph{Representations of Leavitt path algebras,} J. Pure Appl. Algebra {\bf 224} (2020), no. 3, 1297--1319.   
  
  \bibitem{koc2} A. Ko\c{c},  M. \"{O}zayd\i n, \emph{Finite-dimensional representations of Leavitt path algebras},  Forum Math. {\bf 30} (2018), no. 4, 915--928. 
  
\bibitem{ranga} K.M. Rangaswamy, \emph{Leavitt path algebras with finitely presented irreducible representations}, Journal of Algebra {\bf 447} (2016) 624--648. 


\bibitem{vas} L. Vas, \emph{The functor $K_0^{\gr}$ is full and only weakly faithful}, Algebras and Representation Theory {\bf 26} (2023), 2877--2890.




 \end{thebibliography}
 \end{document}